\documentclass[12pt]{amsart}

\usepackage{amsmath, amsthm}
\usepackage{eucal}

\usepackage{amssymb}
\usepackage{amscd}
\usepackage{latexsym}
\usepackage{epsfig}
\usepackage{graphicx}
\usepackage{amsfonts}
\usepackage{psfrag}
\usepackage{color}
\usepackage{url}
\usepackage{cite}
\usepackage[margin=1in]{geometry}

\input xy
\xyoption{all}
\UseComputerModernTips

\numberwithin{equation}{section}
\numberwithin{figure}{section}

  \theoremstyle{plain}
    \newtheorem{theorem}{Theorem}[section]

  \newtheorem{fact}{Fact}[section]
  \newtheorem{proposition}[theorem]{Proposition}
  
  \newtheorem{lemma}[theorem]{Lemma}
  
  \newtheorem{corollary}[theorem]{Corollary}

\theoremstyle{definition}
  \newtheorem{definition}[theorem]{Definition}
  \newtheorem{example}[theorem]{Example}

 \theoremstyle{remark}
  \newtheorem{remark}[theorem]{Remark}
\numberwithin{equation}{section}

\newcommand{\roots}{\roots}
\newcommand{\mat}{\mathbb{M}}

\def\({\begin{math}}
\def\){\end{math}}
\def\Petschub{ X_{w_J, \Pett} }

\def\Hess{{\mathcal{H}(X,H)}}
\def\Pet{{\mathcal{H}(X,\Pett)}}
\def\rr{r_i}

\def\ad{\textup{ad} }
\def\Ad{\textup{Ad } }

\def\dim{\textup{dim}}

\def\V{\mathcal{V}}

\def\ZZ{{\mathbb Z}}

\def\RR{{\mathbb R}}

\def\CC{{\mathbb C}}

\def\m{m}
\def\nn{q}

\def\Pett{\mathfrak P}

\def\c{ c}
\def\dimn{d}
\def\roots{\Phi}

\def\PetSchub{X_{w_J, \Pett}}

\def\a{\alpha}

\def\matr{\mathbb{N}}

\begin{document}
\pagestyle{plain}

\title{Affine pavings of regular nilpotent Hessenberg varieties and intersection theory of the Peterson variety}

\author{Erik Insko}
\address{Dept. of Mathematics\\
Florida Gulf Coast University 
\\ Ft Myers, FL 33695}
\email{einsko@fgcu.edu}

\author{Julianna Tymoczko}
\address{Dept. of Mathematics\\
Smith College\\
Northampton, MA 01063 }
\email{jtymoczko@smith.edu}



\thanks{EI was partially supported by NSF VIGRE grant DMS-0602242. JT was partially supported by NSF grants DMS-0801554 and DMS-1248171, and by a Sloan fellowship.}

\date{\today}
\maketitle

%%%%%%%%%%%%%%%%%%%%%
%  Abstract
%%%%%%%%%%%%%%%%%%%%%

\begin{abstract}
This paper describes a paving by affines for regular nilpotent Hessenberg varieties in all Lie types, namely a kind of cell decomposition that can be used to compute homology despite its weak closure conditions.   Precup recently proved a stronger result; we include ours because we use different methods. 
We then use this paving to prove that the homology of the Peterson variety injects into the homology of the full flag variety.  
The proof uses intersection theory and expands the class of the Peterson variety in the homology of the flag variety in terms of the basis of Schubert classes.  
We explicitly identify some of the coefficients of Schubert classes in this expansion, which is a problem of independent interest in Schubert calculus.
\end{abstract}

\maketitle

\setcounter{tocdepth}{1}
\tableofcontents

%%%%%%%%%%%%%%%%%%%%% SECTION 1 - Intro %%%%%%%%%%%%%%%%%%%%%%%%%

\section{Introduction}

Hessenberg varieties form the nexus between many important areas of mathematics.  
Two parameters characterize a Hessenberg variety: an element $X$ in a Lie algebra $\mathfrak{g}$ and a subspace $H \subset \mathfrak{g}$ whose conditions are given precisely below. 
 De Mari-Shayman defined  Hessenberg varieties in type $A_n$ to describe geometrically the ordered bases that put the $n \times n$ matrix $X$ into {\em Hessenberg form}, 
which is critical for algorithms in numerical analysis \cite{DeMSha88}.  
De Mari-Procesi-Shayman then generalized Hessenberg varieties beyond the original numerical analysis applications to all Lie types \cite{DPS}.  
Hessenberg varieties make important connections between fields like combinatorics \cite{Fulm97,Mbi10}, 
geometric representation theory \cite{Spr76,Fun03}, algebraic geometry and topology \cite{Bri-Car04, InsYon12}, and the quantum cohomology of the flag variety \cite{K,rietsch}, among others.

In this paper, we study the topology of {\em regular nilpotent Hessenberg varieties}, namely when $X$ is a regular nilpotent element of $\mathfrak{g}$.  In type $A_n$ regular nilpotent $X$ are the matrices with a single Jordan block  whose only eigenvalue is zero.  These varieties arose naturally in Peterson's study of the affine Grassmannian \cite{Pet97} and in Kostant's and Rietsch's further work on the quantum cohomology of the flag variety \cite{K,Rie01}.  Through the connection between Peterson varieties and quantum cohomology, Rietsch discovered a set of beautiful new determinantal identities \cite{rietsch}.  

We study cohomology of regular nilpotent Hessenberg varieties in part because of a conjecture of Reiner's.  The homology of regular nilpotent Hesseberg varieties in type $A_n$ can be described as a quotient $\mathbb{Z}[x_1, x_2, \ldots, x_n]/I_H$ that resembles Borel's presentation of the cohomology of the flag variety \cite{Mbi10,MT12}.  The ideal $I_H$ generalizes Tanisaki's ideal, which is an ideal of symmetric functions that defines the cohomology of an important collection of varieties in geometric representation theory called Springer varieties.  Reiner observed that these quotients $\mathbb{Z}[x_1, x_2, \ldots, x_n]/I_H$ are the 
cohomology rings of a family of Schubert varieties studied by Ding and others \cite{Din97}.  He also conjectured that these quotient rings are also the {\em cohomology} rings of the corresponding Hessenberg variety.  A natural follow-up question  asks for the geometric relationship---if any!---between each regular nilpotent Hessenberg variety and the corresponding Schubert variety.  

Others have also tried to identify the cohomology of regular nilpotent Hessenberg varieties.  Brion-Carrell identified the equivariant cohomology ring of each regular nilpotent Hessenberg variety with 
the coordinate ring of a particular affine curve \cite{Bri-Car04}. %[They do all regular nilpotents and call them Peterson varieties].  
Harada-Tymoczko constructed this ring with combinatorial generators and relations \cite{HT09, HT10}, and together with Bayegan-Harada, gave a classical Schubert calculus presentation of the cohomology ring  \cite{BH10}. Bayegan-Harada identified the cohomology ring for one other family of regular nilpotent Hessenberg varieties \cite{BH12}. 
This leaves open the question of providing an explicit description of the equivariant cohomology rings for all regular nilpotent Hessenberg varieties.

This paper has several main results.  First, we give a paving by affine spaces for regular nilpotent Hessenberg varieties in all Lie types in Theorem \ref{paving}.  Pavings are important for several reasons.  This paving is the first step to identify the equivariant cohomology ring since it proves that regular nilpotent Hessenberg varieties are equivariantly formal, a technical condition that dramatically simplifies equivariant cohomology calculations.  Moreover, this paving ensures the existence of geometric fundamental classes which generate the homology of the regular nilpotent Hessenberg varieties, as proved in Corollary \ref{basis}.   Our proof simplifies and generalizes earlier work of Tymoczko \cite{Tym06b}.  Our work is subsumed by recent work of Precup's, which provides a paving by affines for a larger class of Hessenberg varieties \cite{Pre12}.  However, our methods are different so we chose to include the proof in this manuscript.  Peterson gave a similar paving  of a particular regular nilpotent Hessenberg variety, now called the Peterson variety, but into singular affine varieties instead of affine cells; this is an essential piece of Peterson's description of the quantum cohomology rings of  partial flag varieties $G/P$ \cite{K,rietsch}.

In the next part of our paper, we focus on the Peterson variety.   Theorem \ref{linearlyindependent}  gives a direct geometric proof that the inclusion of the Peterson variety into the flag variety 
induces an injection in homology.  This strengthens and geometrically explains an earlier proof of injectivity by Harada-Tymoczko  \cite{HT09}:  we partially decompose the class induced by each piece of the paving of the Peterson variety in terms of Schubert classes in the flag variety.  Our proof relies on classical intersection theory: we count points in the intersections between particular Schubert varieties in the flag variety and the closures of the pieces of the paving of the Peterson variety.   Harada-Tymoczko's purely combinatorial proof depended on an inspired but unmotivated choice of Schubert classes to generate the cohomology of the Peterson variety; our geometric construction explains their choices.   The Schubert varieties we use for our intersections are dual under the Universal Coefficient Theorem to the cohomology classes in Harada-Tymoczko's work, as we prove in Proposition \ref{UCTduality}.  %[Linda---even highlight]

This injection is a subtle topological result, requiring much more than the paving from Section \ref{section:paving}.  Indeed, pavings of subspaces---even very nice pavings of subspaces, compatible with the ambient spaces---may not induce homology injections, not even in examples completely analogous to regular nilpotent Hessenberg varieties.  Consider the circle $S^1$ embedded as the equator of the sphere $S^2$.  One natural CW-structure for $S^2$ consists of the upper hemisphere, the lower hemisphere, and the cells in the CW-decomposition of $S^1$.  By construction, this intersects $S^1$ in the natural CW-structure for $S^1$. %(like the Carrell-Goresky result), 
 However, the image of $H_*(S^1)$ in $H_*(S^2)$ under the map induced by the inclusion $S^1 \hookrightarrow S^2$ is just the class of a point, since the sphere $S^2$ has no odd-dimensional homology.  %Thus, even with particularly nice subpavings of a CW-structure, the homology classes of a subspace need not inject into the homology of the ambient space. 
A family of smooth Hessenberg varieties called regular semisimple Hessenberg varieties provide similar examples.  Regular semisimple Hessenberg varieties have the same Euler number as the flag variety but their Betti numbers are in lower degrees, so the homology of regular semisimple Hessenberg varieties does not inject into the homology of the flag variety.  Note that the cells of the pavings of  regular semisimple Hessenberg varieties are all even-dimensional, unlike in the previous example of spheres.  

This injection is also an important topological result: various strategies to compute cohomology rings, including those due to Goresky-MacPherson \cite{GorMac10} or Harada-Tymoczko \cite{HT10}, rely on a homology injection like we prove.  In forthcoming work, we use this injection together with Goresky-MacPherson's approach to construct a natural  presentation of the $S^1$-equivariant cohomology ring of the Peterson variety as a quotient of a polynomial ring.  This presentation refines a result about the Poincar\'{e} polynomial announced in work of Brion-Carrell and  due to Peterson \cite{Bri-Car04}.  It is much simpler than Harada-Tymoczko's \cite{HT09}. It also highlights the similarities between the Peterson variety and the ambient flag variety.  

Our paper is structured as follows.  In Section \ref{section:paving} we prove that regular nilpotent Hessenberg varieties are paved by affines in all Lie types.
The proof of Theorem \ref{paving} relies on a similar argument to those given in previous paving results of Tymoczko \cite{Tym06,Tym06b}: 
we construct an iterated fiber bundle whose fibers are affine spaces and whose base is inductively proven to be affine.  
The difference from the previous results is a different decomposition of the root spaces, which permits us to generalize to all Lie types.  
Some case-by-case analysis is deferred to tables in the Appendix.  

In Section \ref{3} we prove that the cells described in Theorem \ref{paving} correspond to a basis of homology classes in $H_*(\Hess, \ZZ)$.  
The proof uses a result of Carrell and Goresky, Proposition \ref{CarGor}.
In Section \ref{Peterson} we restrict attention to the Peterson variety, proving that for each Peterson-Schubert homology class there is 
an opposite Schubert variety which intersects it properly in 
Proposition \ref{properintersection}.   

The intersections in Proposition \ref{properintersection} are often not transverse in Lie types other than Type $A_n$.  One might wonder if this failure in transversality is 
due to one of the varieties being singular at the point of intersection.  It follows from Theorem \ref{paving} that the Peterson-Schubert varieties are smooth at the intersection point; the largest cell 
in the paving is an affine neighborhood of the intersection point.  However, it is not obvious
whether the Schubert varieties are singular at the intersection point.  Thus, in Section \ref{section:smooth} we use {\em Kumar's criterion}, 
which describes when a torus-fixed point is smooth in a Schubert variety, to identify an interesting collection of smooth points in certain Schubert varieties in classical Lie types, 
and explain precisely how smoothness fails in exceptional Lie types in Theorem \ref{smooth}.

In Section \ref{classexpansions} we give partial expansions of Peterson-Schubert homology classes in terms of the Schubert basis classes in Theorem \ref{anonzero}.
We then use these expansions to prove that the homology of the Peterson variety injects into the homology of the flag variety in Theorem \ref{linearlyindependent}.
The proof of Theorem \ref{anonzero} uses classical intersection theory and the result of Proposition \ref{properintersection} to partially determine the homology class of each Peterson-Schubert variety.
The proof of Theorem \ref{linearlyindependent} then follows from Theorem \ref{anonzero}.  The partial expansions of the Peterson-Schubert classes indentified in Theorem \ref{anonzero} 
show the push-forward $i_*$ induced by the inclusion map $\iota: \Pet \hookrightarrow G/B$ maps the basis of the homology of the Peterson variety to a linearly independent set in the homology of the flag variety.

\section{Background}

\subsection{Notation} 
Let $G$ be a complex reductive linear algebraic group.  
Fix a pair of opposite Borel subgroups $B$ and $B^{-}$ of $G$ such that $ B \cap B^{-}= T$ is a maximal torus.  
Let $G/B$ denote the flag variety and let $W= N(T)/T$ denote the Weyl group of $G$.  

The subgroup $B$ determines a decomposition of $G$ into double cosets $G = \prod_{w \in W} BwB$ known as the Bruhat decomposition.  
Passing to the quotient $G/B$ one obtains a decomposition of the flag variety $G/B$ into cosets $G/B = \prod_{w \in W} BwB/B$, called {\em Bruhat} or {\em Schubert cells} and often denoted $C_w = BwB/B$.  
The closures $X_w = \overline{BwB/B}$ are the {\em Schubert subvarieties} of the flag variety.  

Let $\mathfrak g, \mathfrak b,$ and $ \mathfrak t$ denote the Lie algebras of $G,B$ and $T$ respectively.   The adjoint action of $G$ on $\mathfrak{g}$ is defined by 
\[Ad(g)(X) = gXg^{-1}\] for any $X \in \mathfrak g$.
Let $\roots$ denote the roots of $\mathfrak g$ and $\roots^+$ the roots corresponding to $\mathfrak b$. 
The set $\roots$ has a partial order defined by the rule that $ \alpha > \beta$ if and only if $\alpha-\beta$ is a sum of positive roots. 
The subalgebra $\mathfrak b$ determines a base $\Delta \subseteq \roots^{+}$ of simple positive roots $\alpha_i \in \Delta$.  
We denote the root space corresponding to $\alpha$ by $\mathfrak g_{\alpha}$, the root subgroup corresponding to $\alpha$ by $U_{\alpha}$, 
and the maximal unipotent subgroups of $B$ and $B^{-}$ by $U$ and $U^{-}$ respectively.   

We may use the subgroup $U_w= \{ u \in U : w^{-1} u w \in U^{-}\}$ to choose a set of coset representatives for the Schubert cell for $w$.     
In fact, we can write $U_w$ as a product of root subgroups.

Define $\roots_w = \{ \alpha \in \roots^{+} : w^{-1} \alpha < 0 \}$.   These are the roots that index $U_w$ in the sense that
\[ U_w = \prod_{\alpha \in \roots_w} U_{\alpha} \]
where the product is taken over any fixed ordering of the roots $\roots_w$ \cite[Section 28.1]{H}.  Moreover, these roots determine the Schubert cell corresponding to $w$.

\begin{lemma}\label{H3}\cite[Theorems 28.3--4]{H}
The Schubert cell $BwB/B$ is homeomorphic to $U_w$.
\end{lemma}

\subsection{Background on Hessenberg Varieties}
In this section we define regular nilpotent Hessenberg varieties and recall the significance of a paving by affines.  

\begin{definition}
A \textbf{Hessenberg space $H$} is a linear subspace 
of $\mathfrak{g}$ such that $ \mathfrak b \subseteq H$ and  $[H,\mathfrak b ] \subseteq H.$
\end{definition}
In other words $H$ is a $\mathfrak b$-submodule of $\mathfrak g $ containing $\mathfrak b$. 

\begin{lemma} \cite[Lemma 1]{DPS}
The datum of a Hessenberg space $H$ is equivalent to the set of roots $M_H \subset \roots^-$ satisfying the following closure condition:
If $\beta \in M_H$ and $\alpha \in \Delta$ is a positive simple root such that $\beta+\alpha \in \roots^-$ then the sum $\beta+\alpha \in M_H$.   

The sets $\{H\}$ and $\{M_H\}$ are related by decomposing each Hessenberg space into root spaces:
\[ H = \mathfrak b \oplus \left ( \bigoplus_{\alpha \in M_H} \mathfrak g _\alpha \right ).\] 
\end{lemma}

A Hessenberg space $H \subseteq \mathfrak{g}$ together with an element of the Lie algebra $\mathfrak{g}$ determine a Hessenberg variety.

\begin{definition} Given an element $X \in \mathfrak{g}$ and a Hessenberg space $H$, the {\em Hessenberg variety} $\Hess$ is defined to be
\[ \Hess = \left \{ gB \in G/B \mid Ad(g^{-1})(X) \in H \right \}.\]
\end{definition} 

Most geometric and topological results about Hessenberg varieties are independent of the choice of $X$ in its adjoint orbit in $\mathfrak{g}$, as described in the following lemma.  (The proof merely notes that conjugating $X$ by an element $g \in G$ is equivalent to translating the Hessenberg variety $\Hess$ by $g$.)

\begin{lemma}\cite[Lemma 5.1]{HT10} 
The Hessenberg variety $\Hess$ is homeomorphic to the Hessenberg variety $\mathcal{H}(Ad(g)X, H)$.
\end{lemma}

A {\em regular nilpotent Hessenberg variety} is a Hessenberg variety $\Hess$ for which the parameter $X$ is a regular (or principal) nilpotent element in $\mathfrak{g}$.  

Jordan canonical form gives a nice representative for each nilpotent orbit in type $A_n$.  Nilpotent orbits are classified for other Lie types, but the classification depends on the type.  However, in all Lie types, the regular nilpotent orbit contains a particularly nice representative which is consistent with Jordan form in type $A_n$.

\begin{fact} \cite[Theorem 4.1.6]{CM93} \label{fact: ColMcG}
For each $i$, choose a simple root vector $E_{\alpha_i} \in \mathfrak{g}_{\alpha_i}$.
The element $X_0 = \sum_{\alpha_i \in \Delta} E_{\alpha_i}$ 
 is in the regular nilpotent adjoint orbit of $\mathfrak{g}$ regardless of the Lie type of $\mathfrak{g}$. \end{fact}

Henceforth in this paper, when we refer to a regular nilpotent Hessenberg variety, we will mean $\mathcal{H}(X_0,H)$.

Given a Schubert cell $C_w$ and Hessenberg variety $\Hess$, we define a \emph{Hessenberg Schubert cell} to be the intersection $C_{w,\mathcal H} = C_w \cap \Hess$.  The \emph{Hessenberg Schubert variety} 
$X_{w, \mathcal H}$ is the closure of the Hessenberg Schubert cell in the Hessenberg variety, namely $X_{w, \mathcal H } = \overline{C_{w,\mathcal H}}$.

Later in the paper, we specialize to the {\em (Dale) Peterson variety}, which is a particular regular nilpotent Hessenberg variety \cite{Pet97, K, rietsch}.  The Hessenberg space for the Peterson variety is 
\[ \Pett = \mathfrak b \oplus \left ( \bigoplus_{\alpha_i \in \Delta^{-} } \mathfrak g_{\alpha_i} \right ) .\]  
When we consider the Peterson variety, we will denote the Hessenberg space by $\Pett$ and the corresponding set of roots by $M_{\Pett}$.   (Note that $M_{\Pett} = \Delta^{-}$.)
We will also refer to Peterson Schubert cells $C_{w, \Pett}$ and Peterson Schubert varieties $X_{w,\Pett}$.

%%%%%%%%%%%%%%%%%%%%% SECTION 1 - Intro %%%%%%%%%%%%%%%%%%%%%%%%%

%%%%%%%%%%%%%%%% SECTION 3 -Paving Hessenbergs %%%%%%%%%%%%%%%%%%%%%%%

\section{Affine Pavings of Hessenbergs} \label{section:paving}

In this section, we prove one of our main results: regular nilpotent Hessenberg varieties are {\em paved by affine spaces} in every Lie type.  This means that regular nilpotent Hessenberg varieties can be partitioned so that each part is homeomorphic to $\CC^d$, subject also to certain closure conditions.  Colloquially, pavings are like $CW$-decompositions but with less restrictive closure relations imposed on the cells.  Most importantly, the closure conditions of a paving by affines are sufficient to guarantee that the cells form geometric generators for the homology $H_*(X,\ZZ)$.   

Our proof uses induction, relying on certain subgroups of the unipotent group whose product gives the whole unipotent group.  We will construct a sequence of partial Schubert cells that approximates more and more finely a Hessenberg Schubert cell, and that terminates in the Hessenberg Schubert cell itself.  We will then prove that each partial Schubert cell in the sequence is an affine bundle over the previous cell in the sequence.

We begin with the definitions crucial to our proof.  We then give a minor technical lemma used in our calculations.  Finally, we provide the construction and proof of the paving by affines for regular nilpotent Hessenberg varieties of arbitrary Lie type.

\subsection{Essential definitions}

We now give the precise definition of pavings and of the unipotent subgroups we use later in this section.

\begin{definition}
A \emph{paving} of an algebraic variety $\V$ is an ordered partition into disjoint cells $\V_0, \V_1, \V_2, \ldots, $ such that every finite ordered union $\cup_{i =0}^j \V_i$ is Zariski-closed in $\V$.  A \emph{paving by affines} is a paving where each cell $\V_i$ is homeomorphic to affine space $\V_i \cong \mathbb C^{d_i}$.  
\end{definition}

If $C$ is a cell in a paving by affines of a variety $X$ then the fundamental class of the closure $\overline{C}$ is a nonzero homology class in $H_*(X,\ZZ)$.  The collection of these fundamental classes form a $\ZZ$-module basis of $H_*(X, \ZZ)$  (see Proposition \ref{CarGor}). 

We now describe the subgroups of the unipotent group that we use.  Every positive root $\beta$ can be written as a positive integer linear combination $\beta = \sum c_i \alpha_i$ of the simple roots.  We define the \emph{height} of the root $\beta$ to be the sum of the coefficients  in this linear combination
 and denote $ht(\beta) = \sum c_i$.  

\begin{definition} The {\em roots of height $i$} are the elements of the set 
\[\roots_i = \{\alpha \in \roots^+: ht(\alpha) = i\}.\]  
The subgroup $U_i$ is defined to be $\prod_{\alpha \in \roots_i}U_\alpha$ where the product is taken over any fixed ordering of $\roots_i$.
\end{definition}

The next lemma gives a key property of the unipotent group.

\begin{lemma}\label{H4} \cite[Proposition 28.1]{H}
The unipotent group $U$ factors as the product   
\[ U = U_{\beta_1}U_{\beta_2}\cdots U_{\beta_k} \] for any fixed ordering of the positive roots $\beta_1, \beta_2, \ldots, \beta_k$. 
\end{lemma} 

If $N$ is the maximal height of a positive root then $U_1 U_2 \cdots U_N = U$ by Lemma \ref{H4}.

\subsection{The adjoint action of $U_i$} 
This section establishes a key computational result using a standard formula for the adjoint action.  We begin with some notation.

For each root ${\gamma}$ we choose a basis for the root space $\mathfrak{g}_{\gamma}$ consisting of a root vector $E_{\gamma} \in \mathfrak{g}_{\gamma}$.  In the appendix we will make specific claims about the structure constants of the group $G$ for which we will need to choose {\em particular} $E_{\gamma} \in \mathfrak{g}_{\gamma}$.  The work in this section is independent of the choice of $E_{\gamma}$ as long as it is fixed at the outset.

For each root subgroup $U_{\alpha} \subseteq U$ we fix the group homomorphism $u_{\alpha}: \mathbb{C} \rightarrow U_{\alpha}$ given by $u_{\alpha}(z) = \exp(zE_{\alpha})$ where $\exp$ is the exponential map \cite[Section 2 of Introduction, %Proposition~0.14, 
Proposition~1.84]{Kna02}

This brings us to our lemma, which begins the technical analysis of the adjoint action of root subgroups $U_{\alpha}$ on root vectors in $\gamma$.

\begin{lemma}\label{H1}  
Let $U_{\alpha} \subseteq U$ be a root subgroup and fix a group homomorphism $u_{\alpha}: \mathbb{C} \rightarrow U_{\alpha}$.  Choose a root vector $E_{\gamma} \in \mathfrak{g}$ for each root $\gamma \in \Phi$.  Fix a root $\beta \in \Phi$.  Then
\[ \Ad(u_\alpha(x))E_\beta = \sum \frac{m_{ j\alpha,\beta}}{j!} x^j E_{\beta+ j\alpha} \]
where the sum is over all $j \geq 0$ such that $\beta+ j \alpha \in \roots$, the structure constant $m_{j\alpha, \beta}$ is a nonnegative integer determined by the choice of Chevalley basis, and $m_{0, \beta}=1$. 
\end{lemma}

\begin{proof}
Recall that  $\Ad(\exp X) = \exp(\ad X) = \sum_{n=0}^{\infty} \frac{(\ad X)^n}{n!}$ as stated in \cite[Proposition 1.91]{Kna02}.  
This means that $\Ad (u_\alpha(x)) = \Ad(\exp (x\cdot E_\alpha)) = \exp ( \ad (x \cdot E_\alpha)) = \sum_{n=0}^{\infty} \frac{(\ad (x \cdot E_\alpha))^n}{n!}.$
It follows that \begin{align*} \Ad(u_{\alpha}(x))(E_\beta) = \exp(\ad(x \cdot E_\alpha))(E_\beta) & = 1+ [x \cdot E_\alpha, E_\beta] + \frac{1}{2}[x \cdot E_{\alpha},[x \cdot E_{\alpha},E_\beta ]] + \cdots \\ 
      &= 1+ \m_{\alpha, \beta} x  E_{\alpha+\beta} + \frac{1}{2} \m_{2\alpha,\beta} x^2 E_{2\alpha+\beta} + \cdots .\end{align*}
Hence $m_{0,\beta} = 1$ and $\Ad(u_\alpha(x))E_{\beta} = \sum \frac{m_{j \alpha,\beta}}{j!} x^j E_{\beta+ j \alpha} $.     \end{proof}

As the previous lemma demonstrates, the operator $\Ad(u_{\alpha}(x_\alpha))$ is not generally linear.  However, by restricting the adjoint action to certain parts of $U$ 
we will be able to exploit an ``almost-linear" phenomenon.

 \begin{lemma}\label{lemma: adjoint calculation}
Fix $N = \sum_{\beta > 0 } n_\beta E_\beta $ to be a sum of positive root vectors and let 
\[u=\prod_{\alpha \in \Phi_i} u_{\alpha}(x_{\alpha}) \in U_i.\]  Then  
\[\Ad(u)(N) - N = \sum_{ht(\gamma)>i} c_\gamma E_\gamma \] 
and in particular for each root $\gamma$ with $ht(\gamma)= i+1$ we have
 \[ c_\gamma = \sum_{\alpha {\textup{ s.t. }}ht(\alpha) = i \textup{ and } \gamma - \alpha \in \Delta } \m_{\alpha,\gamma-\alpha} n_{\gamma-\alpha} x_\alpha. \]
\end{lemma}

\begin{proof}
Lemma \ref{H1} says that $\Ad(u_{\alpha}(x_{\alpha}))$ acts ``like" an unipotent linear operator in the following sense: if $\Ad(u_{\alpha}(x_{\alpha}))(E_{\beta}) = \sum_{\gamma \in \roots^+} c_{\gamma} E_{\gamma}$ then Lemma \ref{H1} specializes in certain cases of $c_{\gamma}$ as
\begin{equation}\label{equation: root subgroup is unipotent}
 c_{\gamma} = \left\{ \begin{array}{ll} 0 & \textup{ if } \gamma < \beta \\
	1 & \textup{ if } \gamma = \beta \\ 
	 m_{ \alpha,\beta} x_{\alpha} & \textup{ if } \gamma = \beta + \alpha. \end{array} \right.
	 \end{equation}
Now take $u = \prod_{\alpha \in \roots_i} u_{\alpha}(x_{\alpha})$ as above and suppose $Ad(u)(E_{\beta}) = \sum_{\gamma \in \roots^+} c_{\gamma} E_{\gamma}$.  Using Equation \eqref{equation: root subgroup is unipotent} repeatedly, we conclude
\begin{equation}\label{equation: complicated unipotent equation}
 c_{\gamma} = \left\{ \begin{array}{ll} 0 & \textup{ if } \gamma < \beta \textup{ or } ht(\beta) <ht(\gamma) < ht(\beta)+i \\
	1 & \textup{ if } \gamma = \beta \\ 
	 m_{\alpha, \beta} x_{\alpha} & \textup{ if } \gamma = \beta + \alpha \textup{ for some } \alpha \in \roots_i. \end{array} \right.
	 \end{equation}
Recall that $\Ad(u)$ acts linearly on $\mathfrak{g}$ in the sense that 
\[\Ad(u)\left( \sum %_{ht(\beta)>i} 
n_\beta E_\beta \right) = \sum %_{ht(\beta)>i} 
n_\beta \left( \Ad(u)(E_\beta) \right).\]  
Write $\Ad(u)(N) - N = \sum_{\gamma} c_\gamma E_\gamma$.  From above, it follows that $c_{\gamma} = 0$ unless $ht(\gamma)>i$.  The only way for both $ht(\gamma) = i+1$ and $\gamma = \beta+ \alpha$ for some $\alpha \in \roots_i$ and $\beta \in \roots^+$ is for $\beta$ to be a simple root.  In that case, by linearity Equation \eqref{equation: complicated unipotent equation} says that $c_{\gamma}$ is as desired.
\end{proof}

The next corollary is an immediate consequence of the previous lemma.

\begin{corollary}\label{corollary: adjoint on regular nilpotent}
Let $X = \sum_{\alpha_j \in \Delta} E_{\alpha_j}$ be a sum of simple root vectors in $\mathfrak{g}$ and let $u \in U$ be a unipotent element.  Then 
\[\Ad(u^{-1})(X) - X = \sum_{\gamma \in \roots^+: ht(\gamma) > 1} c_{\gamma} E_{\gamma}.\]
\end{corollary}

\begin{proof}
Fix an arbitrary ordering of the positive roots and write $u = \prod_j u_{\beta_j}(x_{\beta_j})$ as per Lemma~\ref{H4}.   Let $N_0 = X$ and for each $i = 1, \ldots, |\roots^+|$, let 
\[N_i = \Ad\left(\prod_{j=1}^i u_{\beta_j}(x_{\beta_j})\right)^{-1}(X),\]
so $N_{i+1} =   \Ad\left(u_{\beta_{i+1}}(x_{\beta_{i+1}})^{-1}\right)(N_i)$.
Since $N_i $ is a sum of positive root vectors, the previous lemma guarantees for each $i$ that 
\[\Ad\left(u_{\beta_{i+1}}(x_{\beta_{i+1}})^{-1}\right)(N_i) - N_i = \sum_{\gamma \in \roots^+: ht(\gamma) > 1} c_{\gamma} E_{\gamma}.\]  
(Indeed, the lemma proves more.) We know that $N_0 =  \sum_{\alpha_j \in \Delta} E_{\alpha_j}$.  By induction $N_i$ can be decomposed into root vectors as 
\[N_i = \sum_{\alpha_j \in \Delta} E_{\alpha_j} + \sum_{ht(\gamma)>1} n_{\gamma, i} E_{\gamma}.\]  The claim follows.
\end{proof}

\subsection{Paving by Affines}
 
We now prove the main result of this section: when $X$ is a sum of simple root vectors, the intersection of the Schubert cell $BwB/B$ with the Hessenberg variety $\mathcal{H}(X,H)$ is homeomorphic to affine space.  Moreover,  we give a closed formula for the dimension of the intersection when it is nonempty.

\begin{theorem} \label{paving}
Let $\mathfrak{g}$ be a Lie algebra of arbitrary Lie type.  Choose root vectors $E_{\alpha_i}$ in each simple root space $\mathfrak{g}_{\alpha_i}$ of $\mathfrak{g}$.   If $X = \sum_{\alpha_i \in \Delta} E_{\alpha_i} $ then 
\begin{enumerate}
\item the variety $U_{w,H}= \{ u \in U_w : \Ad(u^{-1})(X) \in w H w^{-1} \}$ is isomorphic to affine space $\CC^\dimn$ of dimension 
$\dimn =|\roots_w \cap w M_H|$ and
\item this set is nonempty if and only if $ w^{-1} (\alpha_j) \in M_H$ for all $j$.
\end{enumerate}
\end{theorem} 

\begin{proof}

Define the subspace $\mathcal{Z}_i \subseteq \left( U_1 \cdots U_{i-1} U_i \right) \cap U_w$ by
\[\mathcal{Z}_i = \left\{u \in (U_1 \cdots U_{i-1} U_i) \cap U_w:  \begin{array}{c} \textup{ if we denote } \Ad(u^{-1})(X) = \sum_{\gamma \in \roots^+} c_{\gamma} E_{\gamma} \textup{ then } \vspace{0.5em} \\\sum_{\gamma \in \roots_1 \cup \roots_2 \cup \cdots \cup \roots_i \cup \roots_{i+1}} c_{\gamma} E_{\gamma} \in wHw^{-1} \end{array} \right\}.\]
Our convention is that if $i=0$ then $\mathcal{Z}_0 \subseteq \{e\}$ where $e$ is the identity element of $U$.  If $N$ is the maximum height of a root in $\roots^+$ then we have $\mathcal{Z}_{N} = U_{w,H}$ since $U_1 \cdots U_{N-1} U_N = U$ and $\roots_{N+1}$ is empty.  

The factorization $(U_1U_2 \cdots U_{i-1}U_i) \cap U_w$ induces a natural surjection $f_i: \mathcal{Z}_i \rightarrow \mathcal{Z}_{i-1}$ according to the rule that $f_i(u_1u_2 \cdots  u_{i-1}u_i) = u_1 \cdots u_{i-1}$.  We will show by induction on $i$ that each $f_i$ is an affine fiber bundle; it will follow that $\mathcal{Z}_N$ is homeomorphic to affine space.

Our inductive claim is that the subspace $\mathcal{Z}_{i}$ is affine, nonempty if and only if the root $ w^{-1} (\alpha_j) \in M_H$ for all $j$, and if nonempty has dimension
\[\sum_{j=1}^i \left( \left| \roots_{j} \cap \roots_w \right| - \left| \roots_{j+1} \cap (wM_H)^c \right| \right).\]
(Our convention is that this sum is zero if $i=0$.)  

The base case is when $i=0$.  In this case $u =e$ and so $\Ad(e^{-1})(X) = X$.  Thus $u \in \mathcal{Z}_0$ if and only if $X \in w H w^{-1}$.  Equivalently, since $X = \sum_{\alpha_i \in \Delta} E_{\alpha_i}$ we have $u \in \mathcal{Z}_0$ if and only if $w^{-1} (\alpha_i) \in M_H$ for all $i$.  If nonempty $\mathcal{Z}_0$ has dimension zero as desired.

Now assume the claim for $\mathcal{Z}_{i-1}$.  Let $u \in \mathcal{Z}_{i-1}$ and consider the set 
\[\left\{u' \in U_{i}: uu' \in (U_1U_{2} \cdots U_{i-2}U_{i-1})U_i \cap U_w \right\}.\]
We will show that the map $f_i: \mathcal{Z}_i \rightarrow \mathcal{Z}_{i-1}$ is an affine fiber bundle whose fiber has dimension $ \left| \roots_{i} \cap \roots_w \right| - \left| \roots_{i+1} \cap (wM_H)^c \right|$.

We know  $\Ad((uu')^{-1})(X) = \Ad(u'^{-1}) \left( \Ad(u^{-1})(X) \right)$.  Denote $\Ad(u^{-1})(X) =  \sum_{\gamma} \nn_{\gamma} E_{\gamma}$.
Since $u \in \mathcal{Z}_{i-1}$ the coefficients $\nn_{\gamma}=0$  if both $ht(\gamma) < i$ and $E_{\gamma} \not \in wHw^{-1}$.  By Lemma \ref{lemma: adjoint calculation}
\[\Ad((uu')^{-1})(X) = \sum_{ht(\gamma) \leq i} \nn_{\gamma} E_{\gamma} + \sum_{ht(\gamma) > i} d_{\gamma} E_{\gamma}\]
where the coefficients $\nn_{\gamma}$ are the same as those in $\Ad(u^{-1})(X)$.  Corollary~\ref{corollary: adjoint on regular nilpotent} showed that $\nn_{\alpha_j} = 1$ for all simple roots $\alpha_j$. Lemma \ref{lemma: adjoint calculation} says that if $ht(\gamma) = i+1$ then 
\begin{equation} \label{equation: eqns with constant term}
d_{\gamma} = \nn_{\gamma}+ \sum_{\alpha {\textup{ s.t. }}ht(\alpha) = i \textup{ and } \gamma - \alpha \in \Delta } \m_{\alpha,\gamma-\alpha} x_\alpha.
\end{equation}

We must now show that the variety parameterized by the set of $x_{\alpha}$ that simultaneously solve the equations 
\begin{equation} \label{equation: system of eqns}
\left\{d_{\gamma} = 0 \textup{ if } \gamma \in \roots_{i+1} \textup{ but } \gamma \not \in wM_H \right\}
\end{equation}
is affine, that the dimension of the solution set is independent of the choice of $u \in \mathcal{Z}_{i-1}$ (in other words, it is independent of $\nn_{\gamma}$ in Equation~\eqref{equation: eqns with constant term}), and compute that dimension.  
We first note that $\m_{\alpha, \alpha_j} \neq 0$ for all types and all simple roots $\alpha_j$ by inspection.

Fix $\gamma \in \roots_{i+1}$ with $w^{-1}(\gamma) \notin M_H$.  We confirm that the variables $x_{\alpha}$ that appear in Equation \eqref{equation: system of eqns} are not identically zero, namely that if $\alpha \in \roots_{i}$ and $\gamma - \alpha =\alpha_j$ is simple then in fact $\alpha \in \roots_{i} \cap \roots_w$.  To show $\alpha \in \roots_w$ we will prove $w^{-1}(\alpha)<0$.  Recall that $w^{-1}$ is a linear functional on the set of roots so 
\[w^{-1}(\gamma) = w^{-1}(\alpha+ \alpha_j) = w^{-1}(\alpha)+w^{-1}(\alpha_j).\] 
Suppose by contradiction that $w^{-1}(\alpha) >0$.  Since $w^{-1}(\alpha_j) \in M_H$ by the inductive hypothesis we have 
\[w^{-1}(\alpha_j) + w^{-1}(\alpha) \in M_H\] 
which contradicts our hypothesis on $\gamma$.  We conclude that $w^{-1}(\alpha)<0$ as desired.

Our next goal is to show that the system of Equations~\eqref{equation: system of eqns} has full rank regardless of the values of $\nn_{\gamma}$ 
from Equation~\eqref{equation: eqns with constant term}.  We will prove a stronger claim.  Let ${\bf v}_{\gamma}$ be the vector whose entries are parametrized by $\alpha \in \roots_i$ and whose $\alpha^{th}$ entry is $m_{\alpha, \gamma-\alpha}$ if $\gamma - \alpha \in \Delta$ and $0$ otherwise.  We will show that for every subset 
\[\Gamma \subseteq \left\{\gamma: \gamma  \in \roots_{i+1} \textup{ but } \gamma \not \in wM_H \right\} \] 
the vectors $\{\bf{v}_\gamma: \gamma \in \Gamma\}$ are linearly independent.

This proof is given in the Appendix.  In classical types, we define an injective function $\rr$ that assigns to each root 
$\gamma \in \roots_{i+1}$ a root  $\rr(\gamma) \in \roots_{i}$ with $\gamma - \rr(\gamma) \in \Delta$.   We then consider the $|\roots_{i+1}| \times |\roots_{i+1}|$ matrix whose $(\gamma, \gamma')$ entry is given by 
$m_{\rr(\gamma'), \gamma - \rr(\gamma')}$ if $\gamma - \rr(\gamma')$ is a root and is zero otherwise. 
Lemma \ref{appendix1} confirms in classical types that this matrix is invertible.  
It follows that any subset of rows from the matrix is linearly independent.  
The vectors $\{{\bf v}_{\gamma}: \gamma \in \Gamma\}$ described above 
may have additional columns but  remain linearly independent.
Lemma \ref{appendix2} proves by explicit calculations that the analogous matrices have maximal rank in the exceptional types.

We conclude that the system of Equations~\eqref{equation: system of eqns} is consistent regardless of the values of $\nn_{\gamma}$.  This system has $\left| \roots_{i+1} \cap (wM_H)^c \right|$ equations in the $ \left| \roots_{i} \cap \roots_w \right|$ variables $x_{\alpha}$ (not all of which need to appear in the system of equations).  The dimension of the solution space is thus
\begin{equation} \label{equation: partial dimension}
 \left| \roots_{i} \cap \roots_w \right| - \left| \roots_{i+1} \cap (wM_H)^c \right|.
 \end{equation}
It follows that the map $f_i: \mathcal{Z}_i \rightarrow \mathcal{Z}_{i-1}$ is an affine fiber bundle with dimension of the fiber given in Equation~\eqref{equation: partial dimension}.  By induction we conclude that $\mathcal{Z}_N$ is an iterated sequence of affine fiber bundles over an affine space.  In other words, the space $\mathcal{Z}_N$ is homeomorphic to affine space.   It is nonempty  if and only if $ w^{-1} (\alpha_j) \in M_H$ for all $j$.  Moreover, if nonempty, the dimension of $\mathcal{Z}_N$ is
\[ \sum_{i=1}^N \left( \left| \roots_{i} \cap \roots_w \right| - \left| \roots_{i+1} \cap (wM_H)^c \right| \right).\]
By convention we know $\roots_{N+1} = \emptyset$ and by the hypothesis that $ w^{-1} (\alpha_j) \in M_H$ for all $j$ we know $\roots_1 \cap  (wM_H)^c = \emptyset$.  It follows that
\[ \sum_{i=1}^N \left( \left| \roots_{i} \cap \roots_w \right| - \left| \roots_{i+1} \cap (wM_H)^c \right| \right) =  |\roots_w| - |\roots^+ \cap (wM_H)^c|  \]
which is $|\roots_w \cap wM_H |$ as desired.
\end{proof}

%%%%%%%%%%%% SECTION 2 - Paving Hessenbergs %%%%%%%%%%%%%%%%%%%%%

%%%%%%%%%%%%%%%% SECTION 3 - Calculating Intersections %%%%%%%%%%%%%%%%%%

\section{Intersections of Peterson-Schubert varieties and Schubert varieties}  \label{3}
In this section we collect results that show that the Hessenberg Schubert varieties correspond to a basis of $H_*(\Hess, \ZZ)$.  
We then restrict our attention to the case of the Peterson variety 
and show that Peterson Schubert varieties intersect certain Schubert varieties properly in a point. These intersections will be used in Section \ref{classexpansions} to give a geometric proof 
that the map $i_*: H_*(\Hess, \ZZ) \rightarrow H_*(G/B, \ZZ)$ is an injection.

\subsection{Torus actions and Bialynicki-Birula cells}
We begin this subsection by recounting some standard facts 
about torus actions and Bialynicki-Birula decompositions of the flag variety.   
Then we use these facts to prove that the Hessenberg Schubert cells identified in Theorem \ref{paving} correspond to generators of $H_*(X,\ZZ)$.
We conclude this section with a useful criterion for identifying which Hessenberg Schubert cells intersect Schubert varieties.      
We state the results used in this section narrowly, specialized to the context in which we need them, but provide references for the interested reader.

Let $T \subset B$ be the maximal torus acting on $G/B$ by left multiplication.  The set of $T$-fixed points in $G/B$ is precisely the set of flags corresponding
to Weyl group elements: \[ G/B^T= \{ wB \in G/B : w \in W \} .\]  
Bialynicki-Birula proved that complete algebraic varieties (like the flag variety $G/B$) with suitable one-dimensional torus actions 
can be partitioned into subspaces according to the fixed point into which  torus orbits flow.  
This partition, known as the Bialynicki-Birula decomposition and described next, 
is a particularly useful torus-invariant decomposition of a variety.

\begin{proposition} \cite[Theorem 4.4]{BB76} Let $\lambda: G_m \rightarrow T$ be a one-parameter subgroup of the torus that acts on a complete scheme 
$X$ with isolated fixed points $\{ p_1,\ldots,p_k \}$.  Define the cells 
\[BB^+_{p_i} =  \left \{ x \in X  \mid \lim_{z \rightarrow 0} \lambda(z)x = p_i  \right \}\]
and 
\[BB^-_{p_i} = \left \{ x \in X  \mid \lim_{z \rightarrow \infty} \lambda(z)x = p_i  \right \}\] 
as the collection of points lying on $G_m$-invariant curves that limit into or out of $p_i$ respectively.
These cells give two canonical decompositions of $X$, namely $X = \coprod BB^+_{p_i}$ and $ X = \coprod BB^-_{p_i}$.  
Moreover, each cell is an affine space
$BB^+_{p_i} \cong \CC^{\ell_i}$ (respectively $BB^-_{p_i} \cong \CC^{m_i}$) and contains exactly one $G_m$-fixed point.
\end{proposition} 

In general, Bialynicki-Birula cells can have extraordinarily complicated topological properties.  
However, if a one-dimensional subtorus acts  generically on $G/B$ 
then in fact the Bialynicki-Birula decomposition agrees with the Schubert cell decomposition for $X=G/B$.

\begin{proposition}\cite[p. 546]{Aky81}  
Suppose $\lambda: G_m  \rightarrow T$ is a one-parameter subgroup that acts on $G/B$ with  $\langle \alpha, \lambda \rangle > 0$ for all $\alpha \in \roots^+$.  
Then   $(G/B)^{G_m} = (G/B)^T$ and the Bialynicki-Birula decomposition \( \{BB^{-}_w: w \in W \} \)   
agrees with the Schubert cell decomposition, in the sense that $BB^{-}_{w} = B^{-}wB/B$ for each $w \in W$.  
\end{proposition}

The previous proposition is stated in terms of the  Bialynicki-Birula minus decomposition and the opposite Schubert decomposition to what we use here.  The proposition also implies that the  Bialynicki-Birula plus decomposition gives the Schubert decomposition we use in this manuscript, namely $BB^{+}_{w} = BwB/B$ for each $w \in W$.

Harada-Tymoczko identified a subtorus $S \subseteq T$ that they proved satisfies $\langle S , \alpha \rangle > 0$ for all $\alpha \in \roots^+$ \cite[Lemma 5.1]{HT10}.  
Thus the Bialynicki-Birula decomposition of $G/B$ with respect to $S$ agrees with the Schubert cell decomposition, in the sense that $BB^+_w = C_w$ for each $w \in W$.  Moreover, Harada-Tymoczko proved that $S$ acts on the regular nilpotent Hessenberg varieties, assuming that $X=X_0$ is chosen as in Fact \ref{fact: ColMcG}.    This means that
 the paving by affines obtained in Section \ref{section:paving}
is a generalization of the Bialynicki-Birula decomposition to the (singular) regular nilpotent Hessenberg varieties.  (Parts of this are implicit in Kostant's work on the Peterson variety \cite[p 64-65]{K}.)
Henceforth, $S \subset T$ will denote the particular one-dimensional torus used by Harada-Tymoczko \cite{HT10}.

Since the fixed point set $\Hess^S$ is finite and each cell $C_{w, \mathcal H} \cong U_{w, \mathcal H}$ is homeomorphic to affine space, 
the following result of Carrell and Goresky can be used to show the closures of these cells correspond naturally to generators of $H_*(X,\ZZ)$.
We restate their result in the weaker form that suffices for our purposes.  
 
\begin{proposition} [Carrell and Goresky \cite{CG83}] \label{CarGor}
Let $Y$ be a compact Kaehler manifold with a holomorphic $\mathbb{C}^*$-action that has nonempty fixed point set.  
Let $X$ be a closed $\mathbb{C}^*$-invariant subspace of $Y$ with finite isolated $\mathbb{C}^*$-fixed points $X^{\mathbb{C}^*} = \{X_1, X_2, \ldots, X_r\}$.  
Denote the Bialynicki-Birula plus decomposition of $X$ with respect to $\mathbb{C}^*$ by $X = \bigcup_{j=1}^r X_j^+$.
If each $X_j^+$ is homeomorphic to affine space $\mathbb{C}^{m_j}$ then the closures $\overline{X_j^+}$ generate $H_*(X, \ZZ)$ freely.  
\end{proposition}

\begin{proof}
 Carrell and Goresky defined a {\em good} action of $\mathbb{C}^*$ to be one whose Bialynicki-Birula plus cells satisfy two conditions:
\begin{itemize}
\item For each $j=1,2,\ldots,r$ there exists a holomorphic map $p_j: X_j^+ \rightarrow X_j$ for which $X_j^+$ is a topologically locally trivial fiber bundle with affine fibers $\mathbb{C}^{m_j}$.
\item If $X_j$ is singular then it has an analytic Whitney stratification satisfying certain additional properties (see \cite[Section 1, Definition]{CG83} for details).
\end{itemize} 
(A priori, there are four conditions defining a good action \cite[Section 1, Definition]{CG83}, but two are satisfied by every Bialynicki-Birula decomposition \cite[Section 2, Remark]{CG83}.)  If each fixed point component $X_j$ is a point then the second condition is vacuous.  If each Bialynicki-Birula plus cell is an affine cell then the first condition is satisfied.  Carrell and Goresky define a map $\mu_j: H_i(X_j) \rightarrow H_{i+2m}(X)$ that factors through $H_{i+2m}(\overline{X_j^+})$ \cite[Section 5, Definition]{CG83}.  They then prove that the map $\bigoplus_j {\mu_j}: \bigoplus_j H_{i-2m_j}(X_j) \rightarrow H_i(X)$ is an isomorphism \cite[Section 6]{CG83}.  In particular, the classes of the closures $\overline{X_j^+}$ generate $H_*(X)$.
\end{proof}

\begin{corollary} \label{basis}
The closures of the cells $C_w \cap \Hess$ freely generate the integral homology $H_*(\Hess, \ZZ)$.  
\end{corollary}

\begin{proof}
The generalized flag variety $G/B$ is a compact Kaehler manifold and the subtorus $S \subseteq T$ is a holomorphic $\mathbb{C}^*$-action with nonempty, 
finite fixed-point set.  Theorem \ref{paving} proved that the intersections $C_w \cap \Hess$ form a paving by affines of the Hessenberg varieties.  
The Schubert cells $C_w$ are also the Bialynicki-Birula plus cells of the $S$-action so the previous proposition applies.
\end{proof}

Unlike Schubert varieties in $G/B$, a {\em Hessenberg Schubert} variety may not be a union of Hessenberg Schubert cells.  
Indeed, the closure of a Hessenberg Schubert cell may intersect with other Hessenberg Schubert cells while not containing them. 
 Nonetheless, the torus $S$ acts on Hessenberg Schubert varieties.

\begin{lemma} \label{lemma:Hess Schubs are S-closed}
Fix a Hessenberg variety that admits an $S$-action and let $wB$ be any $S$-fixed point in $\Hess$.  The action of $S$ on the flag variety $G/B$ restricts to an action of $S$ on the Hessenberg Schubert variety $X_{w, \mathcal H}$.
\end{lemma}

\begin{proof}
For each point $p \in X_{w, \mathcal H}$ there is a sequence of points $p_1, p_2, \ldots$ in $C_{w} \cap \Hess$ with limit $\lim_{i \mapsto \infty} p_i = p$.  For each $s \in S$ and each $i$,  the point $s \cdot p_i$ is in $C_{w} \cap \Hess$ because both $C_{w}$ and $\Hess$ carry $S$-actions.  The torus $S$ acts continuously on $G/B$ so the limit $\lim_{i \mapsto \infty} s \cdot p_i = s \cdot p$.  It follows that $s \cdot p$ is in $X_{w, \mathcal H}$ for each point $p \in X_{w, \mathcal H}$. 
\end{proof}

Since  Schubert cells and Hessenberg varieties are both $S$-invariant, 
we  obtain a useful criterion for identifying their intersections.
\begin{lemma} \label{lemma:fixpt}
 If the intersection 
$X^{v} \cap X_{w, \mathcal H}$ is nonempty 
then it must contain an $S$-fixed point $uB$.
If the intersection 
$X^{v} \cap (C_w \cap \Hess)$ is nonempty 
then it must contain the $S$-fixed point $wB$.  
\end{lemma}

\begin{proof}  All of the spaces in question admit $S$-actions so their intersection also carries an action of $S$. 
If $gB \in C_u$ is contained in this intersection
then the limit $\lim_{z \rightarrow 0} \gamma(z)gB= uB$ must be contained in the intersection because the Schubert cells agree with the Bialynicki-Birula cells of the $S$-action on $G/B$.   This proves the first claim; for the second, take $u=w$.  \end{proof}

%%%%%%%%%%%%%%%%%%%%%%%%%%%%%%%%%%%%%%%%%%%
%
% Section : cells of the Peterson variety
%
%%%%%%%%%%%%%%%%%%%%%%%%%%%%%%%%%%%%%%%%%%%

\section{Cells of the Peterson variety} \label{Peterson}

For the rest of this paper, we specialize from regular nilpotent Hessenberg varieties to Peterson varieties, namely the case when $M_H = \Delta^{-}$.  We begin by summarizing results  from the literature (including this paper), stated for the special case of the Peterson varieties.  We then identify Schubert cells in the flag variety $G/B$ that the Peterson Schubert cells intersect properly.  In subsequent sections, we will determine when the intersections are actually transverse and identify some consequences.

In Section \ref{3} we fixed a particular one-dimensional torus $S \subseteq T$ that is generic in the sense that the $S$-fixed points of $G/B$ are the same as the $T$-fixed points of $G/B$.  We also stated that $S$ acts on each regular nilpotent Hessenberg variety $\Hess$ using the assumption that $X$ has the form of Fact \ref{fact: ColMcG}.  Harada-Tymoczko showed that the $S$-fixed point $wB$ in $G/B$ is contained in the 
regular nilpotent Hessenberg variety $\Hess$ if and only if $w^{-1} X w \in H$  \cite[Lemma 5.1]{HT10}. 

For Peterson varieties, the description of the $S$-fixed points simplifies substantially.  
Fix a subset $J = \{ \alpha_{i_1}, \alpha_{i_2}, \ldots, \alpha_{i_{|J|} } \} \subset \Delta$ and consider 
the Coxeter subgroup $W_J \subseteq W$ generated by the simple reflections $\{s_{i_j}: \alpha_{i_j} \in J\}$.  
Each Coxeter group has a unique maximal element with respect to Bruhat order; 
denote the maximal element of $W_J$ by $w_J$ \cite[Proposition 2.3.1]{BjoBre03}.

\begin{proposition}\cite[Proposition 5.8]{HT10}
An $S$-fixed point $wB \in G/B$ is contained in the Peterson variety $\Pet$ if and only if $w=w_J$ is a maximal element of $W_J$ for some subset $J \subset \Delta$.  
\end{proposition}

We denote the set of $S$-fixed points in $\Pet$ by $\{ w_JB \}$ to emphasize this bijective correspondence between fixed points and parabolic subgroups $W_J \subseteq W$.

\subsection{Proper intersections of Peterson-Schubert varieties and Schubert varieties}

In this subsection we associate to each Peterson Schubert variety $\Petschub$ an opposite Schubert variety $X^{v_J}$ which intersects $\Petschub$ properly in a point.  
We then show that the intersection $\Petschub \cap X^{v_K}$ is empty whenever $|J|=|K|$ but $J \neq K$.
We use the following result repeatedly in what follows.

\begin{proposition} \cite[Chapter 10, Proposition 7]{F} \label{fixedpointsinSchubertvarieties}
The $T$-fixed points in $X_w$ consist of all flags $vB$ with $v \leq w$ in Bruhat order.  The $T$-fixed ponts in $X^u$ consist of $vB$ with $u \leq v$.
\end{proposition}

Let $v_J$ and $u_J$ be the Coxeter elements of the form \[ v_J = s_{i_{|J|}}s_{i_{|J|-1}} \cdots s_{i_1} \mbox{ and } u_{J}= s_{i_1}s_{i_2} \cdots s_{i_{|J|}} \ \text{  where }  i_1 < i_2 < \cdots < i_{|J|}. \]  
The results in this section refer to $v_J$.  Analogous results also hold for $u_J$.  

\begin{lemma}\label{onlypt}
The Peterson Schubert variety $ \Petschub$ intersects the Schubert variety $ X^{v_J}$ in exactly one point and that intersection point is $w_J B$. 
\end{lemma}
\begin{proof} 
We first show that there is exactly one $S$-fixed point in $\Petschub \cap X^{v_J}$.  
The Peterson Schubert variety $\Petschub$ is contained in the Schubert variety $X_{w_J}$ so the $S$-fixed points in $\Petschub$ are a subset of $\{wB \textup{ with }w \leq w_J\}$.  The $S$-fixed points in $\Pet$ all have the form $w_{K} B$ for some $K$ so the $S$-fixed points in $\Petschub$ have the form $w_K B$ with $K\subseteq J$.  
By contrast, the fixed points in $X^{v_{J}}$ are flags $wB$ satisfying $w \geq v_J$. 
Hence, if $wB \in X^{v_{J}} \cap \Petschub$ then $w = w_K$ with $ J \subseteq K \subseteq J$.  
We conclude $w_JB$ is the only fixed point that can be contained in $\Petschub \cap X^{v_J}$.  Since $w_JB \in C_{w_J} \cap \Pet$ and $w_JB \in X^{v_J}$ we conclude that $w_JB$ is the unique $S$-fixed point in $\Petschub \cap X^{v_J}$ as desired.

 Now let $gB$ be a point in the intersection $\Petschub \cap X^{v_J}$. 
The intersection $\Petschub \cap X^{v_J}$ is a closed $S$-invariant subspace of the flag variety by Lemma \ref{lemma:Hess Schubs are S-closed}, so the limits $\lim_{ z \rightarrow \infty} \gamma(z)gB$  and $\lim_{ z \rightarrow 0} \gamma(z)gB$ are contained in $\Petschub \cap X^{v_J}$.  If $gB$ is not an $S$-fixed point then the limits $\lim_{ z \rightarrow \infty} \gamma(z)gB$  and $\lim_{ z \rightarrow 0} \gamma(z)gB$ are distinct in $G/B$, 
and so they are distinct fixed points in $\Petschub \cap X^{v_J}$.  This contradicts the first paragraph.
If $gB$ is a fixed point then 
\[\lim_{ z \rightarrow \infty} \gamma(z)gB = \lim_{ z \rightarrow 0} \gamma(z)gB =gB\] 
and by the previous paragraph $gB=w_JB$.  This proves the claim.
\end{proof}

The next lemma shows that when $|J|=|K|$ then the Peterson-Schubert variety $\Petschub$ intersects the opposite Schubert variety $X^{v_K}$ nontrivially only if $J=K$.  By the previous lemma this 
intersection is the point $w_JB$. 
The proof, which we provide for convenience, follows from Proposition \ref{fixedpointsinSchubertvarieties} and the following fact.

\begin{fact} \cite[p 39]{BjoBre03} \label{BjoBreFact}
The  permutations $w_J \leq w_{K}$ in Bruhat order if and only if $K \subset J$.
\end{fact}

\begin{lemma} \label{empty}
Let $J$ and $K$ be two distinct sets of simple roots in $\Delta$ with $|J|=|K|$.  Then the intersection $ \Petschub \cap X^{v_{K}}$ is empty.   
\end{lemma}

\begin{proof}  In light of Lemma \ref{lemma:fixpt} it suffices to check which $S$-fixed points are in the intersection.  
Suppose $J \neq K$.  The fixed-point $uB$ is contained in $X_{w_J} \cap X^{v_K}$ if and only if $v_K \leq u \leq w_J$.  
However Fact \ref{BjoBreFact} says that $v_K \not \leq w_J$ since $K \not \subseteq J$. 
Hence $X_{w_J} \cap X^{v_K}$ is empty, which implies $\Petschub \cap X^{v_K}$ is empty as well. \end{proof}

\begin{example}
 In type $B_5$, let $J = \{ \alpha_2, \alpha_3, \alpha_4, \alpha_5  \}$ and $K = \{ \alpha_1 , \alpha_2, \alpha_4, \alpha_5\}$. 
Then using the following reduced word decompositions
\begin{align*}
 w_J & = (s_2s_3s_4s_5 s_4 s_3 s_2)(s_3s_4s_5s_4s_3)(s_4s_5s_4)s_5,  \\
 w_K & = (s_1s_2s_1)(s_4s_5s_4s_5),\\
   v_J &  = s_5 s_4 s_3 s_2 \mbox{, and }
v_K   = s_5 s_4 s_2 s_1, 
\end{align*}  one can verify that $v_K \not \leq w_J$ and $v_J \not \leq w_K$.  \qed 
\end{example}

The last tool we need before proving the main result of this section is the following lemma, which specializes the results of Section \ref{section:paving} to the Peterson case to show $\dim \Petschub = |J|$.

\begin{lemma} \label{dimensiontangent}
 The  Peterson Schubert variety $\Petschub = X_{w_J} \cap \Pet$ has dimension $|J|$.
\end{lemma}

\begin{proof}
 The dimension of the intersection $X_{w_J} \cap \Pet$ is $|\roots_{w_J} \cap w_J M_{\Pett}|$ by Theorem \ref{paving}.  
The set of roots $M_{\Pett}$ defining the Peterson variety consists of all negative simple roots.  
As the longest element in $W_J$, the element $w_J$ is the unique element which sends every simple root $ \alpha_{j} \in J $ to a negative simple root, 
and sends no simple root $\alpha_{k} \notin J$ to a negative root \cite[Definition 2.10]{BjoBre03}.  
The intersection $\roots_{w_J} \cap w_J M_{\Pett}$ is thus $ J $, whose cardinality is $|J|$, as desired.
\end{proof}

We now use the previous lemma together with other results from this section to prove that the intersections $\Petschub \cap X^{v_J}$ are proper.  

\begin{proposition} \label{properintersection}
The subvarieties \(\Petschub\) and \(X^{v_J}\) of \(G/B\) intersect properly at the \(S\)-fixed point \(w_JB\).
\end{proposition}

\begin{proof}
The dimension of $G/B$ is $|\roots^+|$
and the dimension of $X_{v_J}$ is $\ell(v_J)=|J|$.
Schubert varieties and opposite Schubert varieties have complementary dimensions, in the sense that  $\dim(X^{v_J})+ \dim (X_{v_J}) = \dim(G/B)$.  It follows that 
 $\dim \ X^{v_J} = |\roots^+| - |J|$.  Lemma \ref{dimensiontangent} showed that $\Petschub$ has dimension $|J|$. This implies that 
\[ \dim (G/B) = \dim (\Petschub) + \dim (X^{v_J}) . \]
Thus the intersection \(\Petschub \cap X^{v_J}\) is proper if and only if it is a disjoint union of finitely many points. 
Lemma \ref{onlypt} showed that the intersection \(\Petschub \cap X^{v_J}\) contains only the point \( w_JB \), which proves the claim. \end{proof}

Insko showed that the intersections in Proposition \ref{properintersection} are always transverse in type $A_n$ \cite{Ins12}.  
However, transversality often fails in all other Lie types.  
The first step in understanding why transversality fails is to identify when one of the varieties is singular at the point of intersection.
In the next section we classify whether the Schubert variety $X^{v_J}$ is smooth or singular at the point $w_JB$.

%%%%%%%%%%%%%%%%%% SMOOTHNESS CRITERION %%%%%%%%%%%%%%%%%%%%%%%%%%%%%%%%%%%
 
\section{Smooth intersections of Peterson-Schubert varieties and Schubert varieties} \label{section:smooth}
In this section, we identify which fixed points $w_JB$ are smooth in Peterson-Schubert varieties and in the Schubert varieties $X^{v_J}$.  
First, we show that $w_JB$ is always a smooth point in  $X_{w_J, \mathfrak P}$.  
We then use a form of Kumar's criterion to prove that $w_JB$ is a smooth point in $X^{v_J}$ exactly when $J$ is a {\em classically-embeddable} subset of the simple roots, 
namely that the subgraph of the Dynkin diagram induced by the vertices labeled by $J$ corresponds to a root system of classical type.

Given \( w,v \) in \( W \), fix a reduced word \( s_{i_1} s_{i_2}  \cdots s_{i_\ell} \) for \( w \).
The set \( (w \roots^{+} \cap \roots^{-}) \) first arose when we described the subgroup $U_w$.   We will need the following classical result \cite[p.14]{Hum90}: the set \( (w \roots^{+} \cap \roots^{-}) \) contains the  \( \ell(w) = \ell \) roots 
 \[ \beta_{i_1} = \alpha_{i_1}, \beta_{i_2} = s_{i_1}(\alpha_{i_2}), \ldots, \beta_{i_\ell} = s_{i_1}s_{i_2} \cdots s_{i_{\ell-1}} (\alpha_{i_\ell}) .\] 

Geometers interpret the next definition as the localization of equivariant Schubert classes at $T$-fixed points in $G/B$ although we use it purely combinatorially in this paper.

\begin{definition} Let $v,w$ be in $W$ and fix a reduced word \( s_{i_1} s_{i_2}  \cdots s_{i_\ell} \) for \( w \).  \emph{Billey's formula} \(p_{v}(w)\) is the polynomial in the roots defined by 
\[ p_v(w) = \sum \beta_{j_1} \beta_{j_2} \cdots \beta_{j_k} \] where this sum runs over 
all subsets \( \{ j_1, j_2 , \ldots, j_k \} \subset \{ 1, 2, \ldots, \ell \} \) 
such that \( s_{i_{j_1}}s_{i_{j_2}}\cdots s_{i_{j_k}}\) is a reduced word for \(v\).   
\end{definition} 

Billey showed that $p_v(w)$ is well-defined and in particular is independent of the reduced word chosen for \(w \) \cite[Theorem2 and Lemma 4.1]{Bil99}.
The following is a version of {\em Kumar's criterion} for which $T$-fixed points  in the Schubert variety $X^v$ are smooth.

\begin{lemma}[Kumar's Criterion] \cite[Corollary 7.2.8]{Billey.Lakshmibai}\label{kumar}
Let $p_{v}(w)$ denote the polynomial from Billey's formula.  
The Schubert variety $X^{v}$ is smooth at $wB$ 
if and only if \begin{equation}  p_{v}(w) = \prod_{\alpha \in \roots^+ \textup{ such that } v \not \leq s_\alpha w} \alpha .\end{equation} 
\end{lemma} 

\begin{remark}
Let $J=\Delta$. If we take $w= w_J = w_0$ and $v = v_J = v_{\Delta}$ then the equation in Lemma \ref{kumar} can be rewritten
 \begin{equation} \label{kumar2}  p_{v_\Delta}(w_0) = \prod_{\alpha \in \roots^+ \textup{ such that } v_J w_0 \not \geq s_\alpha} \alpha .\end{equation}
We use this reformulation in types $B_n,C_n,$ and $D_n$.
\end{remark}

Given a subset $J \subset \Delta$, we decompose $J$ into subsets $J_1,\ldots, J_r$ called \emph{blocks} as follows.     
\begin{definition}
A \emph{block} $J_i$ of $J$ is a maximal subset of $J$ whose roots correspond to a connected component of the Dynkin diagram of the root system.  
The block $J_i$ is  \emph{classically embeddable} if it can be embedded in the Dynkin diagram of a classical root system.  \end{definition}

\begin{remark} \label{blockremark}
If $\alpha \in J_1$, $\beta \in J_2$, and $J_1 \cap J_2$ is empty then $s_{\alpha}(\beta) = \beta$ and $s_{\beta}(\alpha) = \alpha$.  Indeed, if either $s_{\alpha}(\beta) \neq \beta$ or $s_{\beta}(\alpha) \neq \alpha$ then $\alpha$ and $\beta$ are adjacent (or the same) in the Dynkin diagram of the root system.  The roots $\alpha$ and $\beta$ cannot be adjacent if $J_1$ and $J_2$ are {\em maximal}  connected sets of roots, and cannot be the same if $J_1$ is disjoint from $J_2$.
\end{remark}

\begin{example}
In $E_8$, the subset $J = \{ \alpha_1, \alpha_3, \alpha_4, \alpha_6, \alpha_7, \alpha_8 \} \subset \Delta $ can be decomposed 
into two blocks $J_1 =  \{ \alpha_1, \alpha_3, \alpha_4 \} $ and $J_2 = \{\alpha_6, \alpha_7, \alpha_8 \}$.  
However, the subset $K = \{ \alpha_1, \ldots, \alpha_6 \} \subset \Delta $ consists of only one block, which is not classically embeddable.

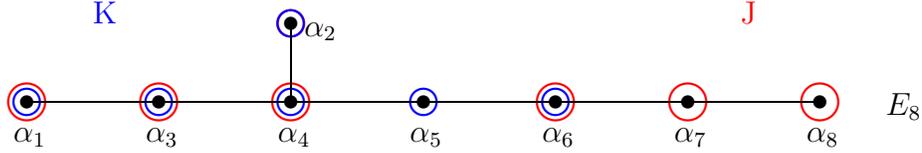
\begin{picture}(100,100)(50,-100)
%%%%%%%%%%%%%%% E8 %%%%%%%%%%%%%%%%%%
\color{red}
\put(365,-40){J}
\thicklines
\multiput(95,-70)(50,0){3}{\circle{14}}
\multiput(295,-70)(50,0){3}{\circle{14}}
\put(195,-40){\circle{10}}

\color{blue}
\put(120,-40){K}

\multiput(95,-70)(50,0){5}{\circle{10}}
\put(195,-40){\circle{10}}

\thinlines

\color{black}
\multiput(95,-70)(50,0){7}{\circle*{5}}
\put(420,-75){$E_8$}
\put(95,-70){\line(1,0){300}}
\put(195,-70){\line(0,1){30}}
\put(195,-40){\circle*{5}}
\put(90,-85){$\alpha_1$}
\put(140,-85){$\alpha_3$}
\put(190,-85){$\alpha_4$}
\put(240,-85){$\alpha_{5}$}
\put(290,-85){$\alpha_{6}$}
\put(200,-45){$\alpha_{2}$}
\put(340,-85){$\alpha_{7}$}
\put(390,-85){$\alpha_{8}$}
\end{picture}
 \qed
\end{example}

The rest of this section is dedicated to proving the next proposition, 
which says that $w_JB$ is smooth in $X^{v_J}$ if and only if every block in $J$ is classically embeddable.  The proof reduces to the single block $J = \Delta$ and then analyzes the different Lie types separately.

\begin{proposition}\label{smooth}
The point $w_JB$ is smooth in $X^{v_J}$ if and only if every block $J_1, J_2, \ldots, J_r \subseteq J$ is classically embeddable.
\end{proposition}

\begin{proof}
The first step is to reduce to the case when $w_J$ is the longest word in a Weyl group and $v_J$ is a Coxeter element.  Suppose \( J \) can be decomposed into smaller blocks
 \( J_1 \cup J_2 \cup \ldots \cup J_r \).  The maximal permutation $w_J$ can be written as the product $w_J = \prod_{i=1}^r w_{J_i}$ and similarly for the permutation $v_J = \prod_{i=1}^r v_{J_i}$.  Choose a block $J_k$.  Every root in $p_{v_{J_k}}(w_{J_k})$ is in $\Phi_{J_k}$.  By Remark \ref{blockremark},  if $\alpha \in J_i$ for $i \neq k$ then $s_{\alpha} \beta = \beta$.  It follows that $s_{\alpha} p_{v_{J_k}}(w_{J_k})= p_{v_{J_k}}(w_{J_k})$ and hence 
 \[ p_{v_J}(w_J)  = p_{v_{J_1}}(w_{J_1}) \cdots p_{v_{J_r}}(w_{J_r}).\] 
Similarly $v_J w_J \geq s_{\alpha}$ if and only if $v_{J_i} w_{J_i} \geq s_{\alpha}$ for some $k$.  In other words $p_{v_J}(w_J)$ satisfies Kumar's criterion if and only if $p_{v_{J_i}}(w_{J_i})$ satisfies Kumar's criterion for each $i$.  We conclude that it suffices to restrict our attention to a single block, namely to the case where \(J = \Delta \) and \( w_J = w_0\).  (Note that the single block $J_i$ may correspond to a root system of a different Lie type than $J$.)

We begin with classical Lie types and then consider the exceptional types.  We will find reduced words for $w_0$ in classical types that satisfy:
\begin{enumerate}
\item we can factor $w_0 = u_1u_2$ where $u_2$ is a reduced word for a Weyl group with one fewer generator than $W$ and of the same Lie type as $W$, 
\item the prefix $u_1$ contains $v_J^{-1}$ as a prefix, namely $u_1 = v_J^{-1}u_3$ for some $u_3$ in $W$, and
\item every reduced word for $v_J$ in $w_0$ actually occurs in $u_1$.
\end{enumerate}
In types $B_n, C_n, D_n$, these conditions can be satisfied simultaneously by one word, while in type $A_n$ we will use two different words for $w_0$, denoted $w_{0,1}$ and $w_{0,2}$.  The table in Figure \ref{figure: longest word classical type} gives these words.  
In each case, the prefix $u_1$ is the word in the first set of parentheses, for instance in type $B_n$ we have $u_1=s_1s_2 \cdots s_{n-1}s_ns_{n-1}\cdots s_1$ and 
\[u_2 = (s_2 \cdots s_{n-1}s_ns_{n-1}\cdots s_2)\cdots (s_{n-2}s_{n-1}s_ns_{n-1}s_{n-2})(s_{n-1}s_ns_{n-1})s_n.\]  

\begin{figure}[h]
 \begin{tabular}{|c |c|} \hline
 Type & Reduced word  \\ \hline 
 \( A_n \) & \( w_{0,1} = (s_1s_{2} \cdots s_{n-1}s_n)(s_1 s_{2} \cdots s_{n-1})\cdots (s_1 s_2) (s_1) \)  \\ 
  & \( w_{0,2} = (s_n s_{n-1} \cdots s_{2}s_1)(s_{n-1} s_{n-2} \cdots s_{1})\cdots (s_n s_{n-1}) (s_n) \)  \\ 
 \( B_n \) &  \( w_0 = (s_1s_2 \cdots s_{n-1}s_ns_{n-1}\cdots s_1)\cdots (s_{n-2}s_{n-1}s_ns_{n-1}s_{n-2})(s_{n-1}s_ns_{n-1})s_n \) \\
 \( C_n \) &  \( w_0 = (s_1s_2 \cdots s_{n-1}s_ns_{n-1}\cdots s_1)\cdots (s_{n-2}s_{n-1}s_ns_{n-1}s_{n-2})(s_{n-1}s_ns_{n-1})s_n \) \\
 \( D_n \) & \( w_0 = (s_1s_2 \cdots s_{n-2}s_ns_{n-1}s_{n-2} \cdots s_2 s_1)\cdots(s_{n-2}s_n s_{n-1}s_{n-2}) s_{n-1}s_n \)  \\
 \hline
  
 \end{tabular}
 \begin{caption}{Reduced words for $w_0$ in classical Lie types} \label{figure: longest word classical type}
 \end{caption}
\end{figure}

We now confirm that these words satisfy Conditions (1)-(3).  The expression for each $w_0$ is a standard form of a reduced word for the longest element of $W$ \cite[Section 3.4]{BjoBre03}.  (Bjorner-Brenti call the expression we use for $w_{0,1}$ in $A_n$ and $w_0$ in $B_n,C_n$ and $D_n$ the \emph{lexicographically first normal form}.  The expression we use for $w_{0,2}$ would logically be called the \emph{lexicographically last normal form}.)
The factorization $w_0=u_1u_2$ satisfies Condition (1) because $u_2$ has the same form as $w_0$ except without the simple reflection $s_1$ in types $B_n, C_n, D_n$ and type $A_n$ with $w_{0,2}$, and without $s_n$ in type $A_n$ with $w_{0,1}$.  %(In fact $u_2$ is the longest word for the Weyl group of the same Lie type as $w_0$ but with one fewer generator in the root system.)
Condition (2) is satisfied by $w_{0,1}$ in type $A_n$ and by $w_0$ in the other types by inspection, since the prefix $u_1$ itself contains $v_J^{-1}$ as a prefix.  (In type $A_n$, the prefix $u_1=v_J^{-1}$ and $u_3=e$.)  We now show that every reduced word for $v_J$ ends in $s_1$.  The word $v_J = s_ns_{n-1}\cdots s_1 $ is a Coxeter element in each classical Lie type and so each reduced word for $v_J$ must satisfy $v_J > s_j$ for all simple reflections $s_j$. If $s_1$ is to the right of $s_2$ in the reduced word for $v_J$ then since $s_1$ commutes with all other $s_j$ we may assume $s_1$ is in the rightmost position, as desired.  If $s_1$ is to the left of $s_2$ then by the same argument we may assume $s_1$ is in the leftmost position in the alternate word $w = v_J$. Canceling $s_1$ on both sides, we conclude that the word $s_1 s_n s_{n-1} \cdots s_2 s_1$ has length $n-1$.  However 
\[s_1 s_n s_{n-1} \cdots s_2 s_1 = s_n s_{n-1} \cdots s_2 s_1 s_2\]
is in lexicographic normal form and so is reduced.  This contradicts our assumption on the length.  Therefore every reduced word of $v_J$ ends with the simple transposition $s_1$ as desired. 

Condition (3) is satisfied by $w_{0,2}$ in type $A_n$ and $w_0$ in the other types because $u_2$ has no instance of $s_1$. 

Next we check that Kumar's criterion holds.

We begin by identifying the roots $\alpha$ with $v_J w_0 \not \geq s_{\alpha}$ using the word $w_{0,1}$ in type $A_n$.  Together, Conditions (1) and (2) tell us that $v_J w_0 \geq w$ for each $w$ in the Weyl group generated by $s_1, s_2, \ldots, s_{n-1}$ in type $A_n$ and generated by $s_2, s_3, \ldots, s_{n}$ in types $B_n, C_n, D_n$.  So the set of roots $\alpha$ for which $v_J w_0 \not \geq s_{\alpha}$ is contained in the set 
\[\{ \alpha \in \Phi^+ \textup{  such that  }  \alpha > \alpha_i \}\]
where $i=n$ in type $A_n$ and $i=1$ in types $B_n, C_n, D_n$.  This is exactly the root system of $A_{n-1}$ so this set is exactly the collection of roots $\alpha$ for which $v_J w_0 \not \geq s_{\alpha}$ in type $A_n$.

In types $B_n, C_n, D_n$ we have $v_J w_0 = (s_{n-i} s_{n-i-1} \cdots s_1) u_2$ where $u_2$ is the longest element in the Weyl group $\langle s_2, s_3, \ldots, s_n \rangle$ and $i=1$ in types $B_n, C_n$ while $i=2$ in type $D_n$.  Thus 
\[v_J w_0 \geq s_{n-j} s_{n-j-1} \cdots s_2 s_1 s_2 \cdots s_{n-j-1} s_{n-j}\]
for any $j \geq 1$ in types $B_n, C_n$ and $j \geq 2$ in type $D_n$.  The reflection $s_{n-j} \cdots s_2 s_1 s_2 \cdots s_{n-j}$ corresponds to the root $\alpha = \alpha_1 + \alpha_2 + \cdots + \alpha_{n-j}$.  It follows that $v_j w_0 \geq s_{\alpha}$ for each root $\alpha = \alpha_1 + \alpha_2 + \cdots + \alpha_{n-j}$ where $j \geq 1$ in types $B_n, C_n$ and $j \geq 2$ in type $D_n$.

Finally, we show that these are all the roots $\alpha$ with $v_J w_0 \geq s_{\alpha}$ in types $B_n, C_n, D_n$.  If $v_J w_0 \geq s_{\alpha}$ and $\alpha > \alpha_1$ then we can write $s_{\alpha} = u s_1 u'$ with $\alpha > \alpha_1$, $u' \in \langle s_2, s_3, \ldots, s_n \rangle$, and $u \leq s_{n-i} s_{n-i-1} \cdots s_2$ for $i=1$ in types $B_n, C_n$ and $i=2$ in type $D_n$.  Every reflection is its own inverse so $us_1 u' = (u')^{-1}s_1 u^{-1}$.  This means the reflection is symmetric about the unique occurrence of $s_1$, namely $u^{-1}=u'$.  Since $u \leq s_{n-i} s_{n-i-1} \cdots s_2$ for $i=1$ in types $B_n, C_n$ and $i=2$ in type $D_n$, we conclude that $\alpha \leq \alpha_1 + \alpha_2 + \cdots + \alpha_{n-i}$ where $i=1$ in types $B_n, C_n$ and $i=2$ in type $D_n$.  This proves that in types $B_n, C_n, D_n$ the set of roots $\alpha$ for which $v_J w_0 \not \geq s_{\alpha}$ is  
\[\left\{ \alpha \in \Phi^+:  \begin{array}{c} \alpha > \alpha_i \textup{ for simple roots } \alpha_i \textup{ labeling } \\ \textup{ {\em at least} two leaves in the Dynkin diagram for } \Phi \end{array} \right\}.\]

We now  identify $p_{v_J}(w_0)$, using the word $w_{0,2}$ in type $A_n$.  Condition (3) implies that
\[p_{v_J}(w_0) = p_{v_J}(u_1).\]
By inspection, we observe that $u_1$ contains $v_J$ as a subword in exactly one way in each of these four Lie types.  In type $A_n$ the word $u_1=v_J$.  
When we evaluate Billey's formula we obtain
\[p_{v_J}(v_J) = \prod_{j=1}^n s_n s_{n-1} \cdots s_{j+1}(\alpha_j) = \prod_{\alpha \textup{ s.t. } \alpha > \alpha_n} \alpha.\]
Similarly, Billey's formula gives
\[p_{v_J}(u_1) = \prod_{j=1}^n s_1 s_{2} \cdots s_{n-i} s_n s_{n-1} \cdots s_{j+1}(\alpha_j)\]
where $i=1$ in types $B_n, C_n$ and $i=2$ in type $D_n$.  Evaluating, we obtain that
\[p_{v_J}(u_1) = \prod_{\alpha  \textup{ s.t. }  \alpha > \alpha_1 \textup{ and another leaf of the Dynkin diagram}} \alpha.\]
In all cases, we conclude that Kumar's criterion holds, namely that
\[p_{v_J}(w_0) = \prod_{\alpha  \textup{ s.t. }  v_J w_0 \not \geq s_{\alpha}} \alpha.\]

\noindent\textbf{Classically embeddable:}
Even in exceptional types, if all of the blocks \( J_1, J_2, \ldots, J_k \) are classically embeddable 
then \( p_{v_J}(w_J) \) is a product of polynomials \( \prod_{i=1}^k p_{v_{J_i}}(w_{J_i}) \) obtained in the above calculations.  
Thus $w_JB$ is smooth in $X^{v_J}$ for any $J$ consisting of classically embeddable blocks.  

\noindent\textbf{Non-classically embeddable:}
Conversely, suppose $J$ contains a  block $J_1$ which is not classically embeddable.  (Figure \ref{non-classically embeddable blocks} gives every example of a non-classically embeddable set $J$.)  We will show that \( p_{v_{J_1} }(w_{J_1}) \) does not satisfy Kumar's criterion, which implies that \(w_JB \) is not smooth in \(X^{ v_J} \).  

We will use the same argument in each exceptional Lie type.  First, we consider $p_{v_J}(w_0)$ as a polynomial in
 the simple roots $\alpha_1, \alpha_2, \ldots, \alpha_n$.  Note that $p_{v_J}(w_0)$ has degree $|J| = \ell(v_J)$ by definition of Billey's formula.  
We find a lower bound on the coefficient of a particular term $\alpha_i^{|J|}$ in $p_{v_J}(w_0)$ and show that the coefficient of $\alpha_i^{|J|}$ is greater
 than what can be obtained as a product of $|J|$ distinct roots.   (The simple root $\alpha_i$ that we use depends on the Lie type.)  To find this lower bound,
 we use the fact that Billey's formula is positive \cite[Equation 4.1 and subsequent discussion]{Bil99}, meaning that each nonzero term in the sum consists of a product of positive roots.  
Thus, it will suffice to find the contributions to $\alpha_i^{|J|}$  from the terms in Billey's formula corresponding to just a few reduced subwords of $w_0$.   (In all cases, we obtained one word 
for $w_0$ from SAGE, and often used the reduced word for $w_0^{-1}$ instead.)

\begin{figure}[h]
\begin{picture}(330,310)(120,-90)
%%%%%%%%%%%%%%% F4 %%%%%%%%%%%%%%%%%%
\multiput(145,155)(50,0){4}{\circle*{5}}
\color{blue}
\thicklines
\multiput(145,155)(50,0){4}{\circle{10}}
\color{black}
\thinlines
\put(315,150){$F_4$}
\put(145,155){\line(1,0){50}}
\put(195,157){\line(1,0){50}}
\put(225,155){\line(-1,1){10}}
\put(225,155){\line(-1,-1){10}}
\put(195,153){\line(1,0){50}}
\put(245,155){\line(1,0){50}}
\put(140,140){$\alpha_1$}
\put(190,140){$\alpha_2$}
\put(240,140){$\alpha_3$}
\put(290,140){$\alpha_{4}$}

%%%%%%%%%%%%%%% E6 %%%%%%%%%%%%%%%%%%
\multiput(145,95)(50,0){5}{\circle*{5}}
\color{red}
\thicklines
\multiput(145,95)(50,0){5}{\circle{10}}
\put(245,125){\circle{10}}
\thinlines
\color{black}
\put(360,90){$E_6$}
\put(145,95){\line(1,0){200}}
\put(245,95){\line(0,1){30}}
\put(245,125){\circle*{5}}
\put(140,80){$\alpha_1$}
\put(190,80){$\alpha_3$}
\put(240,80){$\alpha_4$}
\put(290,80){$\alpha_{5}$}
\put(340,80){$\alpha_{6}$}
\put(255,125){$\alpha_{2}$}

%%%%%%%%%%%%%%% E7 %%%%%%%%%%%%%%%%%%
\multiput(145,40)(50,0){6}{\circle*{5}}

\thicklines
\color{blue}
\multiput(145,40)(50,0){6}{\circle{12}}
\thinlines
\color{black}
\put(410,35){$E_7$}
\put(145,40){\line(1,0){250}}
\put(245,40){\line(0,1){30}}
\put(245,70){\circle*{5}}

\thicklines
\color{blue}
\put(245,70){\circle{12}}
\thinlines
\color{black}

\put(140,25){$\alpha_1$}
\put(190,25){$\alpha_3$}
\put(240,25){$\alpha_4$}
\put(290,25){$\alpha_{5}$}
\put(340,25){$\alpha_{6}$}
\put(255,65){$\alpha_{2}$}
\put(390,25){$\alpha_{7}$}

%%%%%%%%%%%%%%% G2 %%%%%%%%%%%%%%%%%%
\multiput(195,190)(50,0){2}{\circle*{6}}

\color{blue}
\thicklines
\multiput(195,190)(50,0){2}{\circle{10}}

\thinlines

\color{black}
\put(255,185){$G_2$}
\put(195,190){\line(1,0){50}}
\put(195,187){\line(1,0){50}}
\put(195,193){\line(1,0){50}}
\put(225,190){\line(-1,1){10}}
\put(225,190){\line(-1,-1){10}}
\put(190,175){$\alpha_1$}
\put(240,175){$\alpha_2$}

%%%%%%%%%%%%%%% E8 %%%%%%%%%%%%%%%%%%
\multiput(145,-80)(50,0){7}{\circle*{5}}
\thicklines
\color{red}
\multiput(145,-80)(50,0){5}{\circle{14}}
\put(445,-80){\circle{14}}
\thinlines
\color{black}
\put(470,-85){$E_8$}
\put(145,-80){\line(1,0){300}}
\put(245,-80){\line(0,1){30}}
\put(245,-50){\circle*{5}}
\thicklines
\color{red}
\put(245,-50){\circle{14}}
\color{black}
\thinlines
\put(140,-95){$\alpha_1$}
\put(190,-95){$\alpha_3$}
\put(240,-95){$\alpha_4$}
\put(290,-95){$\alpha_{5}$}
\put(340,-95){$\alpha_{6}$}
\put(255,-55){$\alpha_{2}$}
\put(390,-95){$\alpha_{7}$}
\put(440,-95){$\alpha_{8}$}

%%%%%%%%%%%%%%% E8 %%%%%%%%%%%%%%%%%%
\multiput(145,-20)(50,0){7}{\circle*{5}}
\color{blue}
\thicklines
\multiput(145,-20)(50,0){7}{\circle{10}}

\thinlines
\color{black}
\put(470,-25){$E_8$}
\put(145,-20){\line(1,0){300}}
\put(245,-20){\line(0,1){30}}
\put(245,10){\circle*{5}}
\color{blue}
\thicklines
\put(245,10){\circle{10}}
\color{black}
\thinlines
\put(140,-35){$\alpha_1$}
\put(190,-35){$\alpha_3$}
\put(240,-35){$\alpha_4$}
\put(290,-35){$\alpha_{5}$}
\put(340,-35){$\alpha_{6}$}
\put(255, 5){$\alpha_{2}$}
\put(390,-35){$\alpha_{7}$}
\put(440,-35){$\alpha_{8}$}

\end{picture}
\caption{Non-classically embeddable blocks in the exceptional Dynkin diagrams} \label{non-classically embeddable blocks}
\end{figure}
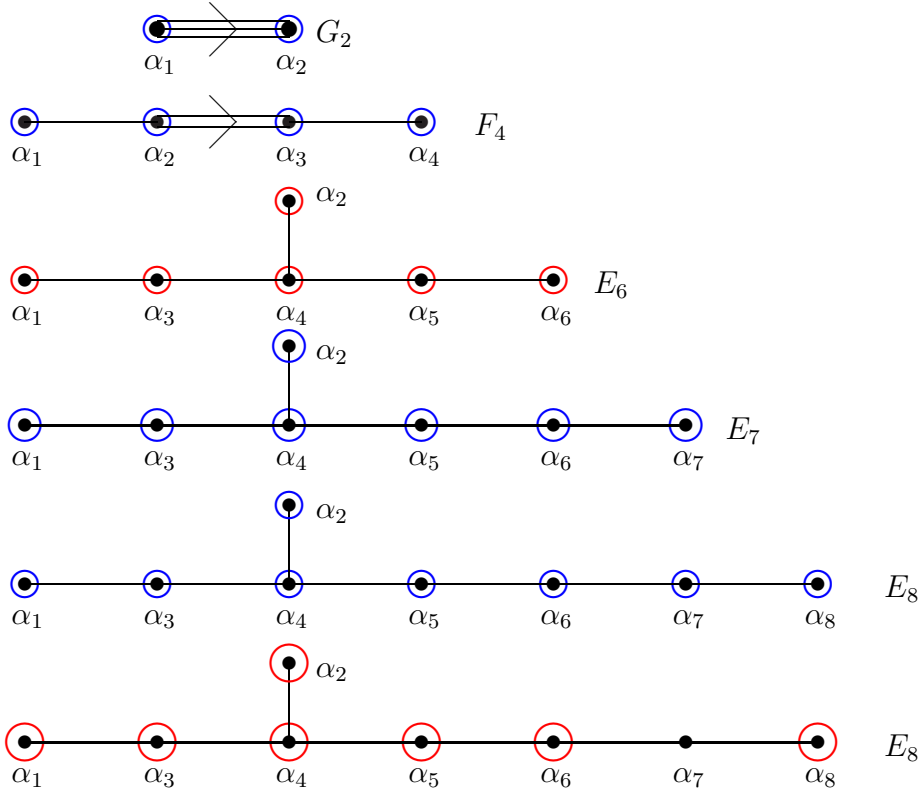

\noindent\textbf{Type \(E_6\):}  Choose the words \(v_J= s_1s_2s_3s_4s_5s_6\) and 
\[w_0 = s_1 s_3 s_4 s_5 s_6 s_2 s_4 s_5 s_3 s_4 s_1 s_3 s_2 s_4 s_5 s_6 s_2 s_4 s_5 s_3 s_4 s_1 s_3 s_2 s_4 s_5 s_3 s_4 s_1 s_3 s_2 s_4 s_1 s_3 s_2 s_1.\]
  Consider the  subwords of $w_0$ underlined below:
 \[\underline{s_1} s_3 s_4 s_5 s_6 \underline{s_2} s_4 s_5 s_3 s_4 s_1 \underline{s_3}s_2 \underline{s_4 s_5 s_6} s_2 s_4 s_5 s_3 s_4 s_1 s_3 s_2 s_4 s_5 s_3 s_4 s_1 s_3 s_2 s_4 s_1 s_3 s_2 s_1, \textup{ and }\]
 \[\underline{s_1} s_3 s_4 s_5 s_6 \underline{s_2} s_4 s_5 \underline{s_3} s_4 s_1 s_3 s_2 \underline{s_4 s_5 s_6} s_2 s_4 s_5 s_3 s_4 s_1 s_3 s_2 s_4 s_5 s_3 s_4 s_1 s_3 s_2 s_4 s_1 s_3 s_2 s_1.\]
We'll compute a lower bound for the coefficient of $\alpha_1^6$ in $p_{v_J}(w_0)$. The coefficient of $\alpha_1$ in each factor in each term of Billey's formula is $1$ so the coefficient of $\alpha_1^6$ in $p_{v_J}(w_0)$ is at least $2$.  However, the coefficient of $\alpha_1$ is at most one in roots of Lie type $E_6$.  Thus the maximum coefficient of $\alpha_1^6$ in a product of six distinct roots is $1$.  

\noindent\textbf{Type \(E_7\):} Choose the words \(v_J= s_1s_2s_3s_4s_5s_6s_7\) and 
\[ \begin{array}{c} w_0 = 
s_7 s_6 s_5 s_4 s_3 s_2 s_4 s_5 s_6 s_7 s_1 s_3 s_4 s_5 s_6 s_2 s_4 s_5 s_3 s_4 s_1 s_3 s_2 s_4 s_5 s_6 s_7 s_1 s_3 s_4 s_5 s_6 s_2 \cdot \\
 s_4 s_5 s_3 s_4 s_1 s_3 s_2 s_4 s_5 s_6 s_2 s_4 s_5 s_3 s_4 s_1 s_3 s_2 s_4 s_5 s_3 s_4 s_1 s_3 s_2 s_4 s_1 s_3 s_2 s_1. \end{array}\]
Billey's formula includes terms corresponding to the following subwords:
\[s_7 s_6 s_5 s_4 s_3 s_2 s_4 s_5 s_6 s_7 s_1 s_3 s_4 s_5 s_6 s_2 s_4 s_5 s_3 s_4 \underline{s_1 s_3 s_2 s_4 s_5 s_6 s_7} s_1 s_3 s_4 \cdots \textup{ and }\]
\[s_7 s_6 s_5 s_4 s_3 s_2 s_4 s_5 s_6 s_7 \underline{s_1} s_3 s_4 s_5 s_6 s_2 s_4 s_5 s_3 s_4 s_1 \underline{s_3 s_2 s_4 s_5 s_6 s_7} s_1 s_3 s_4 \cdots .\]
(Here we use the fact that $s_3s_2=s_2s_3$ because $2$ is only adjacent to $4$ in the Dynkin diagram for $E_7$.)
The simple root $\alpha_7$ occurs in each factor in each term so the term $\alpha_7^7$ occurs with coefficient at least $2$ in $p_{v_J}(w_0)$.  However, the coefficient of $\alpha_7$ is at most one in the roots of type $E_7$ so the coefficient of $\alpha_7^7$ is at most one in a product of seven distinct roots.

\noindent\textbf{Type \(E_8\):}
Fix the following long word for $w_0$ in \(E_8\):
\begin{align*}
w_0 =
&s_8 s_7 s_6 s_5 s_4 s_3 s_2 s_4 s_5 s_6 s_7 s_1 s_3 s_4 s_5 s_6 s_2 s_4 s_5 s_3 s_4 s_1 s_3 s_2 s_4 s_5 s_6 s_7 s_8 s_7 s_6 s_5 \cdot \\
& s_4 s_3 s_2 s_4 s_5 s_6 s_7 s_1 s_3 s_4 s_5 s_6 s_2 s_4 s_5 s_3 s_4 s_1 s_3 s_2 s_4 s_5 s_6 s_7 s_8s_7 s_6 s_5 s_4 s_3 s_2 \cdot \\
 & s_4 s_5 s_6 s_7 s_1 s_3 s_4 s_5 s_6 s_2 s_4 s_5 s_3 s_4 s_1 s_3 s_2 s_4 s_5 s_6 s_7 s_1 s_3 s_4 s_5 s_6 s_2 s_4 s_5 s_3 s_4 s_1 \cdot \\
 & s_3 s_2 s_4 s_5 s_6 s_2 s_4 s_5 s_3 s_4 s_1 s_3 s_2 s_4 s_5 s_3 s_4 s_1 s_3 s_2 s_4 s_1 s_3 s_2 s_1
\end{align*}

Now consider the following two subwords for \(v_J= s_1s_2s_3s_4s_5s_6s_7s_8\):
\[s_8 s_7 s_6 s_5 s_4 s_3 s_2 s_4 s_5 s_6 s_7 s_1 s_3 s_4 s_5 s_6 s_2 s_4 s_5 s_3 s_4 \underline{s_1 s_3 s_2 s_4 s_5 s_6 s_7 s_8} s_7 s_6 s_5 \cdots \textup{   and }\]
\[s_8 s_7 s_6 s_5 s_4 s_3 s_2 s_4 s_5 s_6 s_7 \underline{s_1} s_3 s_4 s_5 s_6 s_2 s_4 s_5 s_3 s_4 s_1 \underline{s_3 s_2 s_4 s_5 s_6 s_7 s_8} s_7 s_6 s_5 \cdots\]
(Again we use the fact that $s_3s_2=s_2s_3$ because $2$ is only adjacent to $4$ in the Dynkin diagram for $E_8$.)

The simple root $\alpha_8$ occurs in each factor in each term, with coefficient $2$ in one factor in each term, so the term $\alpha_8^8$ has coefficient at least $4$.  However, the highest root  in the root system of type $E_8$ is the only root in which $\alpha_8$ appears with coefficient greater than $1$, so the coefficient of $\alpha_8^8$ is at most two in a product of eight distinct roots.  The situation for \( v_J= s_1s_2s_3s_4s_5s_6s_8 \) is similar, since the highest root is obtained by 
\[s_8 s_7 s_6 s_5 s_4 s_3 s_2 s_4 s_5 s_6 s_7 s_1 s_3 s_4 s_5 s_6 s_2 s_4 s_5 s_3 s_4 {s_1 s_3 s_2 s_4 s_5 s_6 s_7 \alpha_8},\]
which appears as a factor in the terms corresponding to both 
 \[s_8 s_7 s_6 s_5 s_4 s_3 s_2 s_4 s_5 s_6 s_7 s_1 s_3 s_4 s_5 s_6 s_2 s_4 s_5 s_3 s_4 \underline{s_1 s_3 s_2 s_4 s_5 s_6} s_7 \underline{s_8} s_7 s_6 s_5 \cdots \textup{   and }\]
\[s_8 s_7 s_6 s_5 s_4 s_3 s_2 s_4 s_5 s_6 s_7 \underline{s_1} s_3 s_4 s_5 s_6 s_2 s_4 s_5 s_3 s_4 s_1 \underline{s_3 s_2 s_4 s_5 s_6} s_7 \underline{s_8} s_7 s_6 s_5 \cdots\]

\noindent\textbf{Type \(F_4\):}
Let \( v_\Delta = s_1 s_2 s_3 s_4 \) and \( w_0 = s_4s_3s_2s_3s_1s_2s_3s_4s_3s_2s_3s_1s_2s_3s_4s_3s_2s_3s_1s_2s_3s_1s_2s_1\). Billey's formula has several terms, including those corresponding to the subwords
\[s_4s_3s_2s_3\underline{s_1s_2s_3s_4}s_3s_2s_3s_1s_2s_3s_4s_3s_2s_3s_1s_2s_3s_1s_2s_1,\]
\[s_4s_3s_2s_3\underline{s_1s_2s_3}s_4s_3s_2s_3s_1s_2s_3\underline{s_4}s_3s_2s_3s_1s_2s_3s_1s_2s_1,\]
and 
\[s_4s_3s_2s_3\underline{s_1s_2}s_3s_4s_3s_2s_3s_1s_2\underline{s_3s_4}s_3s_2s_3s_1s_2s_3s_1s_2s_1.\]
By calculating the corresponding factors in Billey's formula and expanding, we see that each term contributes at least one occurrence of $\alpha_1^4$ to the final polynomial, so the coefficient of $\alpha_1^4$ in $p_{v_\Delta}(w_0)$ is at least $3$.  However, only one root $c_1 \alpha_1 + c_2 \alpha_2 + c_3 \alpha_3 + c_4 \alpha_4$ in the root system for $F_4$ has coefficient $c_1 > 1$, namely the highest (long) root $2 \alpha_1+3 \alpha_2 + 4 \alpha_3 + 2 \alpha_4$.  Thus the coefficient of $\alpha_1^4$ is at most $2$ in any product of four distinct roots in this root system.

\noindent\textbf{Type \(G_2\):}
Let \( v_\Delta = s_1 s_2 \) and \( w_0 = s_1 s_2 s_1 s_2 s_1 s_2 \).  The terms in Billey's formula that correspond to the subwords
\[s_1 s_2 \underline{s_1 s_2} s_1 s_2    \textup{  and  } \underline{s_1 s_2} s_1 s_2 s_1 s_2\]
are $(2\alpha_1 + 3 \alpha_2)(\alpha_1 + 2\alpha_2)$ and $\alpha_1 (\alpha_1+\alpha_2)$ respectively.  Thus the coefficient of $\alpha_1^2$ in $p_{v_J}(w_0)$ is at least $3$.  However, there is exactly one root of Lie type $G_2$ in which the coefficient of $\alpha_1$ is greater than $1$, namely the highest (long) root $2 \alpha_1 + 3 \alpha_2$.  So the maximum coefficient of $\alpha_1^2$ in a product of two distinct roots of Lie type $G_2$ is $2$.
 \end{proof}
%%%%%%%%%%%%%%%% SECTION 3 - Calculating Intersections %%%%%%%%%%%%%%%%%%

%%%%%%%%%%%%%%%% SECTION 4 - Schubert class expansions %%%%%%%%%%%%%%%%%%

\section{Expanding Peterson Schubert classes in terms of Schubert classes}\label{classexpansions}
Let $i_*: H_*(\Pet, \ZZ) \rightarrow H_*(G/B, \ZZ)$ denote the push-forward
of the natural inclusion map $i: \Pet \rightarrow G/B$.
In this section we use intersection theory to determine certain coefficients $a_u$ in the expansion of each Peterson-Schubert class in terms of 
Schubert classes \[ i_*([\Petschub ]) = \sum a_u [X_u] .\] In Theorem \ref{anonzero}, we start by proving that certain coefficients $a_{v_J}$ are nonzero and others $a_{v_K}$ are zero. 
 Theorem \ref{linearlyindependent} uses this information to prove that the push-forward $i_*$ is injective.    
 
 Insko showed that some of the relevant Peterson-Schubert classes and Schubert classes intersect transversally in type $A_n$.  He then used this to conclude that the $a_{v_J}$ in fact equal $1$
 \cite{Ins12}. We conjecture that in the other classical types, the coefficients $a_{v_J}$ are either $1$ or $2$.  More precisely, 
consider the vertices associated to $J$ in the Dynkin diagram for the root system.  
If these vertices are contained in a subdiagram isomorphic to the type $A$ Dynkin diagram, then we conjecture that the coefficient $a_{v_J}$ is one; 
otherwise, we conjecture that the coefficient $a_{v_J}$ is two.
 
\begin{theorem} \label{anonzero}
Let \( i_*([\PetSchub]) \) denote the class in \( H_*(G/B, \ZZ ) \) of a 
Peterson Schubert variety \( \PetSchub . \)  
Let \( [X_w] \) denote the fundamental class of a Schubert variety in \(  H_*(G/B, \ZZ). \)  
If we write \[ i_*([ \PetSchub]) = \sum a_{u} [X_u] \] 
then the coefficients $a_{v_J}$ satisfy  
\begin{enumerate}
\item  \( a_{v_J} \neq 0 \) and 
\item  \( a_{v_K} =0 \) for any \( K \neq J \) with $|K|=|J|$.
\end{enumerate}      
\end{theorem}

\begin{proof}
For each Peterson Schubert class $[X_{w_J,\mathfrak P}]$ in $H_*(\Pet,\ZZ)$, the push-forward $i_*([X_{w_J, \mathfrak P}])$ defines a fundamental class in \( H_*(G/B, \ZZ ) \). 
Similarly, every Schubert variety \(X_w \subseteq G/B \) corresponds to a 
fundamental class \( [ X_w ] \in H_*(G/B, \ZZ ) \) 
and these classes form an orthonormal additive \( \ZZ \)-basis of \(H_*(G/B, \ZZ ) \) \cite[p. 160 Equation (1)]{F}. 

Since the flag variety is a compact orientable differentiable variety, it satisfies Poincar\'e duality. 
Thus we may identify $H_*(G/B,\ZZ)$ with $H^*(G/B,\ZZ)$ and identify the cup product of cohomology classes as 
the proper intersection of their Poincar\'e dual homology classes. Denote the intersection product of the two classes $[X]$ and $[Y]$ by $[X]\cdot [Y]=m[Z]$, 
where the varieties $X$ and $Y$
intersect properly in the subvariety $Z$ and $m$ is the positive integer intersection multiplicity \cite[Appendix B, Equation 31]{F}.  

The closed irreducible subvarieties \( \Petschub \) 
and \( X^{v_J} \) intersect properly in the point $[w_JB]$ by Proposition \ref{properintersection}, so the product  \( i_*([ \Petschub ])\cdot [X^{v_J}] = b_{v_J} [w_JB]\) 
for some integer multiplicity \( b_{v_J} > 0 \) \cite[Appendix B, Equation 31]{F}.  Furthermore $[X_{u}] \cdot [X^{u}] = 1$ and $[X_{u}]\cdot [X^{v}]=0$ if $u \neq v$ because Schubert classes form an orthonormal basis. 
Hence the coefficient of
$[X_{v_J}]$ in the expansion $i_*([\PetSchub]) = \sum a_{u} [X_u]$ must be the nonzero quantity $b_{v_J}$.
 Lemma \ref{empty} says that \( \Petschub \cap X^{v_K} = \emptyset \) when $K \neq J$ and $|K|=|J|$ so we conclude \( i_*([ \Petschub ])\cdot [X^{v_J}] = 0.\)
\end{proof}

Finally we show that the fundamental classes of Peterson-Schubert
varieties form a linearly independent set when expanded in terms of the fundamental classes
of Schubert varieties.  

\begin{theorem} \label{linearlyindependent}
For any Lie type, let $i: \Pet \hookrightarrow G/B$ denote the (proper) embedding 
of $\Pet$ as a closed subvariety of $G/B$.
The push-forward 
\[i_*: H_*(\Pet ; \mathbb Z) \rightarrow H_*(G/B; \mathbb Z)\]
is an injection.  
\end{theorem}

\begin{proof}
Every Peterson-Schubert fundamental class $i_*([\PetSchub])$  can be written as a $\ZZ$-linear combination of the Schubert classes \[ i_*([\PetSchub]) = \sum a_u [X_u] .\]
Fix any $k$ with $0 \leq 2k \leq 2n=\dim_{\RR} (\Pet)$ and consider the matrix of 
\[i_{2k} : H_{2k}(\Pet; \ZZ) \rightarrow H_{2k}(G/B; \ZZ)\]
 with respect to the bases of Peterson-Schubert classes and ordinary Schubert classes, respectively.   Theorem \ref{anonzero} implies that the coefficient $a_{v_J} \neq 0$ in the expansion for $i_*([\PetSchub])$ and $a_{v_J} =0$ in the expansion of $i_*([X_{w_K, \mathfrak P }])$ for $K \neq J$ with $|K|=|J|$. 
Consider the submatrix  of $i_{2k}$ whose columns and rows correspond to $\{ [\PetSchub] \}$ and $ \{ [X_{v_J}] \}$ for all $J$ with $|J|=k$.  This submatrix is a diagonal $b_{2k} \times b_{2k}$ matrix with nonzero entries $a_{v_J}$ along the diagonal, where $b_{2k}$ is the $2k$-th Betti number of \( H_*(\Pet; \ZZ) \). This implies that the map $i_{2k}$ is a nonsingular linear transformation 
for all $k$ with $0 \leq 2k  \leq \dim_{\RR}  \Pet$.  
Hence the map $i_*$ is an injection. \end{proof}

In fact, the intersection theory argument provides strictly more information than the injection of Theorem  \ref{linearlyindependent} since it identifies specific coefficients in the nonsingular matrix
This injection is equivalent to a cohomology surjection, as we show in general in the next result and specifically for Peterson varieties in the subsequent claim. 

\begin{proposition}\label{UCTduality}
Let $Y$ be a topological space with a subspace $X \subseteq Y$ so that neither $X$ nor $Y$ have any odd-dimensional (integral) homology.  For each $i$, suppose that $\{X_j: j \in J_i\}$ form a basis for $H_i(X, \mathbb{C})$ and that $\{Y_k: k \in K_i\}$ form a basis for $H_i(Y,\mathbb{C})$.  Consider the map induced by the inclusion $\iota: X \hookrightarrow Y$ on homology.  Let $c_{j,k}$ denote the coefficients of the expansion $\iota_*(X_j) = \sum_k c_{j,k} Y_k$ in terms of the basis $Y_k$ for $H_i(Y,\mathbb{C})$.  If for each $i$ there are a collection of classes $\{Y_k: k \in \mathcal{I}_i \}$ such that the matrix $(c_{j,k})_{j \in J_i, k \in \mathcal{I}_i}$ is invertible then the classes $\{\iota^*(Y_k): k \in \bigcup_i \mathcal{I}_i\}$ form a basis for the cohomology $H^*(X,\mathbb{C})$ under the map $\iota^*: H^*(Y,\mathbb{C}) \rightarrow H^*(X,\mathbb{C})$.
\end{proposition}

\begin{proof}
Consider the Universal Coefficient Theorem for cohomology.  Neither $X$ nor $Y$ have odd-dimensional (ordinary integral) homology by hypothesis.  Thus the Universal Coefficient Theorem reduces to the isomorphism
\[H^i(X,\mathbb{C}) \cong Hom(H_i(X,\mathbb{Z}), \mathbb{C}).\]
This isomorphism is natural with respect to the chain map induced by the inclusion 
\[\iota: X \hookrightarrow Y\] 
so for each $i$ we get a commutative diagram
\[\begin{array}{ccc}
H^i(X,\mathbb{C}) &\cong& Hom(H_i(X,\mathbb{Z}), \mathbb{C}) \\
\uparrow \iota^* & & \uparrow \iota^* \\
H^i(Y,\mathbb{C}) &\cong& Hom(H_i(Y,\mathbb{Z}), \mathbb{C}) 
\end{array}\]
For each $j$, let $X_j^*$ denote the dual to the basis class $X_j$, namely the function defined by $X_j^*(X_{j'})=\delta_{j,j'}$, and similarly for the duals $\{Y_k^*\}$.  The map 
\[\iota^*: Hom(H_i(Y,\mathbb{Z}), \mathbb{C}) \rightarrow Hom(H_i(X, \mathbb{Z}), \mathbb{C})\]
is defined by $\iota^*(f) = f \circ \iota_*$.  Evaluating at $Y_k^*$ gives 
\[\iota^*(Y_k^*) (X_j)= Y_k^* \circ \iota_*(X_j) = Y_k^* \left(\sum c_{j,k'} Y_{k'} \right)\]
where $c_{j,k'}$ is the coefficient of the term $Y_{k'}$ in the expansion of $\iota_*(X_j)$ in terms of the basis $\{Y_k\}$ for $H_i(Y,\mathbb{C})$. 
 This evaluates to $c_{j,k}$ because $Y_k^*$ is a linear functional.  The submatrix $(c_{j,k})_{j \in J_i, k \in \mathcal{I}_i}$ of these coefficients is invertible by 
hypothesis so the image of the classes $\{\iota^*(Y_k^*): k \in \mathcal{I}_i\}$ forms a basis for $Hom(H_i(X,\mathbb{Z}), \mathbb{C})$.  Using the Universal Coefficient Theorem, 
we conclude $\{\iota^*(Y_k): k \in \mathcal{I}_i\}$ form a basis for $H^i(X,\mathbb{C})$ as well.  
The images $\iota^*(Y_k)$ are homogeneous because each $Y_k$ is a geometric class induced from an irreducible variety, and because the chain maps induced by $\iota$ preserve degree. 
 It follows that the set $\{\iota^*(Y_k): k \in \bigcup_i \mathcal{I}_i\}$ forms a basis for $H^*(X, \mathbb{C})$ as desired.
\end{proof}

The next result confirms that in fact the Peterson Schubert classes satisfy the conditions in the previous claim.  

\begin{proposition}
The images of the Schubert classes $\{\iota^*(X_{v_J}):  J \subseteq \Delta\}$ in $H^*(G/B, \mathbb{C})$ span $H^*(\Pet, \mathbb{C})$.
\end{proposition}

\begin{proof}
Neither $\Pet$ nor $G/B$ have odd-dimensional cohomology since they each have pavings by complex affine cells by Theorem \ref{paving}.  
Corollary \ref{basis} showed that the fundamental classes of the Peterson Schubert varieties form a basis for $H_i(\Pet, \mathbb{C})$, respectively the ordinary Schubert classes and $H_i(G/B,\mathbb{C})$.  When we expand $\iota_i(\PetSchub)$ into ordinary Schubert classes, we obtain
\[\iota_i(\PetSchub) = \sum_{w} c_{J,w}X_w.\]
Theorem \ref{anonzero} states that $c_{J,v_{J'}}$ is nonzero only when $J=J'$.  In other words, the matrix $(c_{J, v_{J'}})_{J, J' \subseteq \Delta s.t. |J|=|J'|=i}$ is an invertible diagonal matrix.  This shows that the Peterson varieties satisfy the conditions of Proposition \ref{UCTduality}.
\end{proof}

\section{Thanks}
The authors are thankful to Sam Evens, Megumi Harada, Nicholas Teff, and aBa Mbirika for many helpful conversations during this project.

\section{Appendix} \label{appendix}

In this appendix we prove that the rows of the matrix $\mat_i$ defined in Equation \eqref{equation: eqns with constant term} are linearly independent in all Lie types.  Our strategy is first to consider a matrix with all of the rows of $\mat_i$ as well as some additional rows, and second to restrict to the square submatrix formed by certain columns of that larger matrix.  We will then prove that the resulting square matrix is invertible.  It follows that the rows of the larger matrix are linearly independent, and hence the rows of $\mat_i$ are linearly independent. 

The rest of this appendix is structured as follows.  We will define the square submatrix in Section \ref{section: square submatrix}.  For most of our arguments, we actually only need to know that certain entries of the submatrix are nonzero; however, in a few cases, we will need explicit values for these entries.  Section \ref{section: structure constants of submatrix} collects all the information that we need about the entries of this submatrix.  Section \ref{section: classical lie types} proves the claim for classical Lie types, while the final section proves the claim by inspection in the exceptional Lie types.

\subsection{The square submatrix $\matr_i$ and the injection $\rr:\roots_{i} \rightarrow \roots_{i-1}$}\label{section: square submatrix}

The matrix $\mat_i$ was defined from a particular linear system of equations given in Equation \eqref{equation: eqns with constant term}.  In this appendix, we instead study a larger matrix defined as follows.
For each root $\gamma \in \roots_{i}$ and $\alpha \in \roots_{i-1}$ the entry in row $\gamma$ and column $\alpha$ is given by 
\[\begin{array}{cl}
m_{\alpha, \gamma - \alpha} & \textup{  if }  \gamma - \alpha \in \Delta \textup{ and} \\
0 & \textup{  if }  \gamma - \alpha \not \in \Delta.\end{array}
\]
The entry $m_{\alpha, \gamma - \alpha}$ is an integer defined in the next section.  (The heights of $\gamma$ and $\alpha$ differ by exactly one, so the difference $\gamma- \alpha$ is a root exactly if the difference is a simple root.)  

To construct the submatrix $\matr_i$ we define an injective function $\rr:\roots_{i} \rightarrow \roots_{i-1}$ for each $i \geq 2$ that assigns to 
the root $\gamma \in \roots_i$ a unique root $\rr(\gamma) \in \Phi_{i-1}$ with $\gamma - \rr(\gamma) \in \Delta$.  
The submatrix $\matr_i$ consists of the columns indexed by roots in the image $\textup{Im}(\rr)$.   Lemma \ref{appendix1} proves that $\matr_i$ is invertible in classical Lie types.  Lemma \ref{appendix2}  proves the claim in exceptional types by explicit computations.

We begin by listing the roots in each Lie type according to the height of the root, and choosing an order on roots of a given height. 
 Figure \ref{figure: classical roots and ribbons} gives the roots in classical types; the exceptional roots are listed later in the appendix, and
 are taken from Plotkin and Vavilov's work \cite{PloVav96, Vav04}. To save space, we write the roots in \emph{string Dynkin form}, so that
the positive root $\alpha = \c_1 \a_1 +\c_2 \a_2  + \cdots + \c_n \a_n$ is listed in the table as
 $\alpha = \c_1 \c_2\cdots \c_n$. 
For instance, the root $\alpha = \a_2 + 2 \a_3+2 \a_4$ is listed as $0122$.

The columns of these tables partition the roots into sets with two important properties:
 \begin{enumerate} 
  \item each column contains exactly one root $ \alpha \in \Phi_i$ of height $i$ which appears in row $i$, and 
  \item the root in row $i+1$ of column $j$ is obtained from the root in row $i$ of column $j$ by adding a simple positive root. 
 \end{enumerate}
 The columns in exceptional types are labeled by the index of the simple root in the first row.

\begin{definition}
Define the function $\rr:\roots_{i} \rightarrow \roots_{i-1}$ by the rule that if $\gamma$ is in column $j$ and row $i$ then $\rr(\gamma)$ is the unique root in column $j$ and row $i-1$.  
\end{definition}

For example in type $A_n$ we have $\rr(\a_1+\a_2) = \a_1$ and $\rr(\a_2+\a_3) = \a_2$.  More generally, in type $A_n$ the function is defined by 
\[\rr(\alpha_i+\alpha_{i+1} + \alpha_{i+2} + \cdots +\alpha_{i+k-1} +  \alpha_{i+k}) = \alpha_i+\alpha_{i+1} + \alpha_{i+2} + \cdots +\alpha_{i+k-1}.\]
(The function is not as easily described in other Lie types.)

We now formally define $\matr_i$ to be the matrix whose rows 
are indexed by $\roots_{i}$ and whose columns are indexed by the elements in $\roots_{i-1}$ of the form $\rr(\gamma')$ for some $\gamma' \in \roots_{i}$.  In particular, the entries of $\matr_i$ are indexed by pairs $\gamma, \gamma' \in \roots_i$ and the entry in row $\gamma$ and column $\rr(\gamma')$ given by
\[\begin{array}{cl}
m_{\rr(\gamma'), \gamma - \rr(\gamma')} & \textup{  if }  \gamma - \rr(\gamma')  \in \Delta \textup{ and} \\
0 & \textup{  if } \gamma - \rr(\gamma') \not \in \Delta.\end{array}
\]

\subsection{Structure constants $m_{\alpha, \gamma}$ and the entries of $\matr_i$} \label{section: structure constants of submatrix}  By construction, the entries $m_{\alpha, \beta}$ are structure constants in the Lie algebra.  The matrix $\matr_i$ is generally very sparse, in the sense that almost all entries are zero.  In most cases $\matr_i$ is a permutation of a diagonal matrix, so we will be able to prove that $\matr_i$ is invertible simply using one fact:

\begin{proposition}\cite[Proposition 8.4(d)]{H2}  \label{proposition: structure constant nonzero} %\cite[Theorem 4.2.1]{Car89}
If $\alpha, \beta$ are roots and $\alpha+\beta$ is a root then $m_{\alpha, \beta}$ is nonzero.
\end{proposition}

In some cases we will need more information about the values of $m_{\alpha, \beta}$.  We collect that information in this section.

By choosing root vectors $E_{\gamma} \in \mathfrak{g}_{\gamma}$ appropriately, we may assume that $m_{\alpha, \beta}$ are positive for a particular set of pairs of roots $(\alpha', \beta')$ (called {\em extraspecial} roots) that determine all of the structure constants $m_{\alpha, \beta}$  \cite[Proposition 4.2.2]{Car89}.  %This is how Plotkin and Vavilov build their tables of structure constants in the exceptional types \cite{PloVav96, Vav04}.

\begin{figure}
\begin{tabular}{|l|l|l|l|l|l|} \hline
 \multicolumn{6}{|c|}{Type $A_n$}   \\ \hline 
%               &  1         &  2           & $\cdots$   &  n-1          &  n            \\ \hline
 $\Phi_1$      & $1000\cdots000$ & $0100\cdots000$ & $\cdots$ & $0000\cdots010$  & $0000\cdots001$   \\
 $\Phi_2$      & $1100\cdots000$ & $0110\cdots000$ & $\cdots$ & $0000\cdots011$  &                   \\
 $\Phi_3$      & $1110\cdots000$ & $0111\cdots000$ & $\cdots$ &                  &                   \\
$\vdots$& $\vdots  $      & $\vdots$        & $\cdots$ &                    &          \\ 
 $\Phi_{n-1}$    & $1111\cdots110$ & $0111\cdots111$ &          &                    &          \\
 $\Phi_n$      & $1111\cdots111$ &                 &          &                    &          \\ \hline

\end{tabular}

\begin{tabular}{|l|l|l|l|l|l|} \hline
 \multicolumn{6}{|c|}{Type $B_n$} \\ \hline 
%               &  1            &  2           & $\cdots$ &  n-1         &  n            \\ \hline
 $\Phi_1$      & $1000\cdots000$  & $0100\cdots000$ & $\cdots$ & $0000\cdots010$ & $0000\cdots001$  \\
 $\Phi_2$      & $1100\cdots000$  & $0110\cdots000$ & $\cdots$ & $0000\cdots011$ &                  \\
 $\Phi_3$      & $1110\cdots000$  & $0111\cdots000$ & $\cdots$ & $0000\cdots012$ &                  \\
$\vdots$         &$\vdots  $        & $\vdots$        & $\cdots$ &                 &                  \\ 
 $\Phi_{n-1}$    & $1111\cdots110$  & $0111\cdots111$ &     $\cdots$     &                 &                  \\
 $\Phi_{n}$      & $1111\cdots111$  & $0111\cdots112$ &    $\cdots$      &                 &                  \\ 
 $\Phi_{n+1}$    & $1111\cdots112$  & $0111\cdots122$ &  $\cdots$        &                 &                   \\
$\vdots$         & $\vdots$         & $\vdots$        &          &                 &                   \\
$\Phi_{2n}$      & $1112\cdots222$  & $0122\cdots222$ &          &                      &        \\
$\Phi_{2n-2}$     & $1122\cdots222$ &                               &          &                     &        \\
$\Phi_{2n-1}$    & $1222\cdots222$ &                               &          &                     &        \\ \hline
 
\end{tabular}

\begin{tabular}{|l|l|l|l|l|l|} \hline
 \multicolumn{6}{|c|}{Type $C_n$} \\ \hline 
%               &  1            &  2           & $\cdots$ &  n-1         &  n            \\ \hline
 $\Phi_1$      & $1000\cdots000$  & $0100\cdots000$ & $\cdots$ & $0000\cdots010$ & $0000\cdots001$  \\
 $\Phi_2$      & $1100\cdots000$  & $0110\cdots000$ & $\cdots$ & $0000\cdots011$ &                  \\
 $\Phi_3$      & $1110\cdots000$  & $0111\cdots000$ & $\cdots$ & $0000\cdots021$ &                  \\
$\vdots$       &$\vdots  $        & $\vdots$        & $\cdots$ &                 &                  \\ 
 $\Phi_{n-1}$  & $1111\cdots110$  & $0111\cdots111$ &    $\cdots$      &                 &                  \\
 $\Phi_n$      & $1111\cdots111$  & $0111\cdots121$ &      $\cdots$    &                 &                  \\ 
 $\Phi_{n+1}$    & $1111\cdots121$  & $0111\cdots221$ &   $\cdots$       &                 &                   \\
$\vdots$& $\vdots$         & $\vdots$        &          &                 &                   \\
$\Phi_{2n}$      & $1122\cdots221$  & $0222\cdots221$ &          &                      &        \\
$\Phi_{2n-2}$    & $1222\cdots221$ &                               &          &                     &        \\
$\Phi_{2n-1}$    & $2222\cdots221$ &                               &          &                     &        \\ \hline
 
\end{tabular}

\begin{tabular}{|l|l|l|l|l|l|l|l|} \hline
 \multicolumn{7}{|c|}{Type $D_n$} \\ \hline 
%        &  1         &  2           & $\cdots$ &  n           &  n-2            &  n-1  \\ \hline
 $\Phi_1$      &$1000\cdots0000$ & $0100\cdots0000$ & $\cdots$ & $0000\cdots0100$ & $0000\cdots0001$   &$0000\cdots0010$   \\
 $\Phi_2$      &$1100\cdots0000$ & $0110\cdots0000$ & $\cdots$ & $0000\cdots0110$& $0000\cdots0101$ &        \\
 $\Phi_3$      &$1110\cdots0000$ & $0111\cdots0000$ & $\cdots$ & $0000\cdots0111$     &$0000\cdots1101$ &    \\
$\vdots$&$\vdots  $       & $\vdots$         & $\cdots$ &  & $\vdots$         &   \\ 
 $\Phi_{n-1}$    &$1111\cdots1110$ & $0111\cdots1111$ & $\cdots$ & &  $1111\cdots1101$  &  \\
 $\Phi_n$      &$1111\cdots1111$ & $0111\cdots1211$ &       $\cdots$   &                  &   &      \\ 
 $\Phi_{n+1}$    &$1111\cdots1211$ & $0111\cdots2211$ &  $\cdots$       &                  &   &     \\
$\vdots$&   $ \vdots  $   & $\vdots$ &          &                  &   &  \\
 $\Phi_{2n-5}$   & $1112\cdots2211$& $0122\cdots2211$ &          &                  &   &  \\
 $\Phi_{2n-4}$   & $1122\cdots2211$&                  &          &                  &   &  \\
 $\Phi_{2n-3}$   & $1222\cdots2211$&                  &          &                  &   &     \\ \hline
\end{tabular}
\caption{Roots in classical types and the function $\rr:\roots_{i} \rightarrow \roots_{i-1}$} \label{figure: classical roots and ribbons}
\end{figure}

We will need structure constants for simply-laced root systems, which can be recursively derived from a few relations \cite[Section 15]{PloVav96}:
\begin{itemize}
\item For each pair of roots $\alpha, \beta \in \Phi^+$
\begin{equation}\label{equation: negating coefficient}
m_{\alpha, \beta} = -m_{\beta, \alpha}.
\end{equation}
\item Suppose $\alpha_i+ \beta \in \Phi^+$ and there is no $j< i$ such that $\alpha_i+ \beta = \alpha_j+ \overline{\beta}$ for a root $\overline{\beta} \in \Phi^+$.  Then 
\begin{equation}\label{equation: first recursion formula}
m_{\alpha_i,\beta} =1.
\end{equation}
\item Suppose $\alpha, \beta, \alpha+\beta$ are all in $\Phi^+$ and that $i$ is the minimal index such that
$\alpha+\beta = \alpha_i+ \overline{\beta}$ for some root $\overline{\beta} \in \Phi^+$.   
Then the structure constant is \cite[Equation 14]{Vav04}: 
\begin{equation} m_{\alpha,\beta} = %\begin{cases} - m_{\alpha_i, \beta - \alpha_i}m_{\beta- \alpha_i, \alpha}  & \text{ if } \beta - \alpha_i \in \Phi^+ \\
m_{\alpha_i,\alpha-\alpha_i}m_{\alpha-\alpha_i, \beta}   \text{ if } \alpha- \alpha_i \in \Phi^+. %\end{cases} 
\label{recursion}  \end{equation} 
\end{itemize}

The structure constants we calculate  all have the form $m_{\alpha, \alpha_j}$ for some simple root $\alpha_j \in \Delta$ so we need no other cases of this formula.  

The root $\alpha_i$ in Equation \eqref{recursion} is the simple root of smallest index in $\alpha$ unless $\alpha - \alpha_i$ is not a root---as for instance in the root $\alpha_{n-2}+\alpha_{n-1}+\alpha_n$ in type $D_n$.

\subsection{Classical Lie types}\label{section: classical lie types} We now prove the matrices $\matr_i$ are invertible in classical Lie types. % by showing that the square sub-matrix $\matr_i$ is nonsingular.

\begin{lemma} \label{appendix1}
For each $i$ with $2 \leq i \leq n$ in classical Lie type, the matrix $\matr_i$  is invertible.
\end{lemma}
\begin{proof}
We start with types $A_n, B_n$, and $C_n$.  Suppose that $\gamma_j \in \roots_i$ is in the $j^{th}$ row of the table and $\beta_k \in \textup{Im}(\rr)$ is in the $k^{th}$ column of the table.  Also suppose that the $k^{th}$ column is to the left of the $j^{th}$ column, namely $k<j$. 
 By inspection we see  $\beta_k > \alpha_k$ and $\gamma_j$ is in the span of $\alpha_j, \alpha_{j+1}, \cdots, \alpha_n$.  
In particular $\gamma_j \not > \alpha_k$.  Hence $\gamma_j - \beta_k$ is not a simple root %and only one simple root $\alpha_j$ for which $\beta_k + c \alpha_j = \beta_j$ is possible.  The former means that 
and so the $(j,k)$ entry of $\matr_i$ is zero if $k<j$.  %and the latter means that $m_{\rr(\beta_j), \beta_k}=0$ unless 
%Define a total order on the set of positive roots by the rule that $\alpha \prec \beta$ if either $ht(\alpha) < ht(\beta)$ or both $ht(\alpha) = ht(\beta)$ and the integer representing the string Dynkin form of $\alpha$ is less than that of $\beta$. 
It follows that  the matrices $\matr_i$ are upper-triangular when the roots indexing rows and columns are ordered from leftmost to rightmost as in Figure \ref{figure: classical roots and ribbons}.  The entries along the diagonal are $m_{\rr(\gamma), \gamma - \rr(\gamma)}$ which are nonzero  
by Proposition \ref{proposition: structure constant nonzero}.  Thus the determinant of each matrix $\matr_i$  is nonzero in types $A_n$, $B_n$, and $C_n$ as desired.

The argument for type $D_n$ is slightly different when $i=2$ than for all other $i \geq 3$.  In this case,  reorder the columns of $\matr_2$ so that they are indexed by $\alpha_1$, $\alpha_2$, $\ldots$, $\alpha_{n-3}$, $\alpha_n$, $\alpha_n-2$, and reorder the rows in the same way, so that if $\gamma$ indexes column $j$ then $\rr^{-1}(\gamma)$ indexes row $j$.  The resulting matrix is upper-triangular, so $\det(\matr_2)$ is nonzero.

Now let $i \geq 3$ in type $D_n$.  To show that the determinant of $\matr_i$ is nonzero we partition the rows of the matrix $\matr_i$ and then show that in one subset of $k$ rows, exactly $k$ columns have nonzero entries.  The determinant of $\matr_i$ is thus the product of the minors corresponding to the $k \times k$ submatrix and the complementary $(n-k) \times (n-k)$ submatrix.  (In fact, we repeat this argument twice in our proof.)

The first block consists of rows and columns indexed by  $(\gamma, \rr(\gamma') )$ with $ \gamma, \gamma' \geq  \cdots 01211$.
The second consists of rows and columns indexed by  $(\gamma,\rr( \gamma'))$ with $\gamma, \gamma' \leq 11\cdots 1$.

To begin, consider  the rows indexed by a root $\gamma$ with $\gamma \geq  \cdots 01211$. 
 If $\gamma = \cdots 012 \cdots 2 11$  or $\gamma = \cdots 0 11 2 \cdots 2 11$  
then $\rr(\gamma)$ is the only root in the image $Im(r_i)$ that is less than $\gamma$. Hence there is exactly one nonzero entry in this row.
 Otherwise, the row indexed by $\gamma$ has exactly 2 nonzero entries.  These entries are in the columns indexed by :
\begin{enumerate}
\item $\rr(\gamma)$ and
\item $\gamma-\alpha_i$ where $i$ is the smallest index for which $\gamma-\alpha_i$ is a root. 
\end{enumerate}
The entry indexed by  $(\gamma, \rr(\gamma))$ lies on the diagonal and the entry indexed by $(\gamma, \gamma- \alpha_j)$ is just to the right of the diagonal in $\matr_i$.  

We have shown two things.  First, the matrix $\matr_i$ has the form $\left[ \begin{array}{c|c} A & B \\ \cline{1-2} 0 & D \end{array}\right]$ where $D$ is
 the square submatrix whose rows and columns are indexed by pairs $(\gamma, \rr(\gamma') )$ with roots $ \gamma, \gamma' \geq  \cdots 01211$ and $A$ is the square
 submatrix with rows and columns indexed by $(\gamma,\rr( \gamma'))$ where $\gamma, \gamma' \leq 11\cdots 1$.  Thus $\det (\matr_i) = \det(A) \det(D)$.  
Second, the matrix $D$ is upper-triangular with nonzero elements along the diagonal.  Hence $\matr_i$ is invertible if and only if $A$ is invertible.

Next we consider the square submatrix of $\matr_i$ with rows and columns indexed by $(\gamma, \rr(\gamma'))$ where  $\gamma, \gamma' \leq 11 \cdots 1$.  

Suppose $\gamma' \leq 1 \cdots 100$.  Using Equations \eqref{equation: negating coefficient}, \eqref{equation: first recursion formula}, and \eqref{recursion}, we see that
\[ m_{\gamma', \alpha_j} = \begin{cases} -1 & \text{ if } j \text{ is the smallest index with } \gamma' > \alpha_j \\
                           1 & \text{ if } j \text{ is the biggest index with } \gamma' > \alpha_j
                          \end{cases}\]
When $\gamma' =\cdots 01$ or $\gamma' = \cdots 10$ we have $m_{\gamma', \alpha_j} = -1$ for all $\alpha_j$ with $\gamma' + \alpha_j \in \roots$.  

Together, these determine the square submatrix of $\matr_i$ indexed by  $(\gamma, \gamma - \rr( \gamma'))$ where $\gamma, \gamma' \leq 11\cdots 1$.  For ease of notation, we permute the rows and columns to move the first row and first column of the submatrix to the last row and last column.  The resulting submatrix has the following form: 
\[ \begin{pmatrix} 1&-1 & & &  & & & \\ 
                   & 1 &-1 &  & & & & \\ 
                   & & \ddots &\ddots &  & & & \\ 
                   & &  &\ddots &  \ddots & & & \\ 
                   & &  & & 1 &-1 & & \\ 
                   & & & &  & 1& -1& 0\\
                   & & & &  &0 & -1& -1\\ 
                   & & & &  &  1& 0& -1\\  
   \end{pmatrix}\]
where the last three rows are indexed by roots of the form $\cdots 10$, $\cdots 11$, and $\cdots 01$, and the last three columns are indexed by roots of the form $\cdots 100$ , $\cdots 10$, and $\cdots 01$.  By the same argument as before, the determinant of this matrix is the product of the bottom-right $3 \times 3$ minor with the top-left minor.  Inspection shows that both of these minors are nonzero.  It follows that $\det(\matr_i)$ is nonzero, as desired.

 \end{proof}  

\subsection{Exceptional Lie types}

We now prove that the matrices $\mat_i$ have linearly independent rows in the exceptional Lie types. 
 
\begin{lemma} \label{appendix2}
For each $i$ with $2 \leq i \leq n$ in exceptional Lie types, the matrix $\mat_i$  has maximal rank.
\end{lemma}

To prove that the matrices $\mat_i$ have maximal rank in exceptional types we have listed the  matrices $\mat_i$ below 
and identified, where necessary, a $m \times m$ square submatrix of $\mat_i$ with nonzero determinant (where $m$ is the number of rows in $\mat_i$.) 
We use the structure constants for $F_4$ given by Plotkin and Vavilov \cite[p. 96]{PloVav96} 
and the structure constants for $E_\ell$ ($\ell=6,7,8$) given by Vavilov \cite[p 1526--1546]{Vav04}.

We will compute the determinant of the matrix $\mat_{i}$ if the matrix is square; if not, we will specify columns and compute the determinant of the submatrix with those columns.  
All of the matrices in these tables are sparse, 
in the sense that the determinant \[\det A = \sum_{\sigma \in S_n} \prod_{j=1}^n a_{j,\sigma(j)}\] has few nonzero terms $\prod_{j=1}^n a_{j,\sigma(j)}$.  In fact, many of these matrices have a unique nonzero term in their determinant.  In these cases, we state that the determinant is nonzero because it has a unique term.  Only in the cases where there is more than one nonzero term do we compute the determinant, and then we list it below the matrix.

Finally, we omit any matrices with two rows because for those cases, any two columns are linearly independent by the following lemma.

\begin{lemma} \label{2x2}
If $\mat_i$ is a $2 \times 2$ matrix then it is invertible.  
\end{lemma}
\begin{proof}
Let $\beta$ and $\gamma$ be the two roots of height $i+1$.
Let $r(\beta)$ and $r(\gamma)$ be the roots of height $i$.
By definition of the function $r$ we have
\begin{itemize}
 \item $\beta - r(\beta) = \alpha_i$
\item $\gamma - r(\gamma) = \alpha_j$
\end{itemize}
where $\alpha_i$ could be equal to $\alpha_j$.
This implies that the matrix has nonzero diagonal entries $m_{r(\beta),\alpha_i}$ and $m_{r(\gamma),\alpha_j}$.  If both  non-diagonal entries are zero, then the matrix has full rank.

Now suppose one of the non-diagonal entries is nonzero, say $\gamma-r(\beta)= \alpha_k$.  
We will show $\beta-r(\gamma)$ is not a root and therefore the other non-diagonal entry is zero.  

First note that $\alpha_i \neq \alpha_k$ because $\beta-r(\beta) =\alpha_i$ and $\gamma-r(\beta) = \alpha_k$ but $\gamma \neq \beta$.  
Next note that $\alpha_j \neq \alpha_k$ because $\gamma-r(\gamma) = \alpha_j$ and $\gamma-r(\beta) = \alpha_k$ but $r(\gamma) \neq r(\beta)$.  
Finally we calculate 
\begin{align*} \beta - r(\gamma) &= \beta - (\gamma-\alpha_j) \\ &= r(\beta)+\alpha_i - (r(\beta)+\alpha_k)+\alpha_j \\ & = \alpha_i -\alpha_k + \alpha_j \end{align*}
We conclude that $\beta-r(\gamma)$ is not a root and hence $\matr_i$ is a triangular matrix.  
\end{proof}

 \subsection{Types $G_2$ and $F_4$} 
In type $G_2$ each matrix $\mat_i$ is $1 \times 1$.  The one entry in the matrix is nonzero by Proposition \ref{proposition: structure constant nonzero}.  

In type $F_4$ we only show $\mat_i$ for $i \leq 5$.
This is because the matrix $\mat_i$ has only two nonzero columns if $i>5$;  by Lemma \ref{2x2}, those $\mat_i$ have full rank. 

{\tiny
\begin{tabular}{|l|l|l|} \hline
 \multicolumn{3}{|c|}{Type $G_2$}    \\ \hline 
               &  I     &  II        \\ \hline
$\Phi_1$       & $10$      & $01$      \\
$\Phi_2$       & $11$      &           \\
$\Phi_3$       & $21$      &           \\ 
$\Phi_4$       & $31$      &           \\
$\Phi_5$       & $32$      &           \\ \hline
\end{tabular}
}

{\tiny
\begin{tabular}{|l|l|l|l|l|} \hline
 \multicolumn{5}{|c|}{Type $F_4$}    \\ \hline 
             &  1       &  3    &  2   &  4      \\ \hline
 $\Phi_1$      & $1000$      & $0010$   & $0100$  & $0001$     \\
 $\Phi_2$      & $1100$      & $0011$   & $0110$  &            \\
 $\Phi_3$      & $1110$      & $0111$   & $0120$  &            \\ 
 $\Phi_4$      & $1120$      & $1111$   & $0121$  &            \\
 $\Phi_5$      & $1220$      & $1121$   & $0122$  &            \\ 
 $\Phi_6$      & $1221$      & $1122$   &         &            \\ 
 $\Phi_7$      & $1231$      & $1222$   &         &            \\ 
 $\Phi_8$      & $1232$      &          &         &            \\
 $\Phi_9$      & $1242$      &          &         &            \\
 $\Phi_{10}$   & $1342$      &          &         &            \\
 $\Phi_{11}$   & $2342$      &          &         &            \\  \hline
\end{tabular}

\begin{minipage}[b]{2.5 in}
\begin{tabular}{|c | c | c | c |c | c| }
 \hline 
 \multicolumn{5}{|l|}{Matrix at height 1} \\ \hline \hline
    & I & II  & III & IV \\ \hline   
  I       & 1   & -1   & 0   & 0 \\  \hline
  II      & 0   & 1   & -1  & 0	 \\  \hline
  III     & 0   & 0   & 1  & -1  \\  \hline
\end{tabular}

Det of (I,II, III) has one nonzero term.
\end{minipage}
\begin{minipage}[b]{2.5 in}
\begin{tabular}{|c | c | c | c |c | }
 \hline 
 \multicolumn{4}{|l|}{Matrix at height 2} \\ \hline \hline
    & I & II  & III  \\ \hline   
  I     & 1   & -1   & 0    \\  \hline
  II      & 0   & -2   & 0 	 \\  \hline
  III     & 0   & 1   & -1	   \\  \hline
\end{tabular}

Det has one nonzero term.
\end{minipage}

\begin{minipage}[b]{2.5 in}
\begin{tabular}{|c | c | c | c |c | }
 \hline 
 \multicolumn{4}{|l|}{Matrix at height 3} \\ \hline \hline
    & I & II  & III  \\ \hline   
  I       & -2   & -1   & 0    \\  \hline
  II      & 0   & 1   & -1 	 \\  \hline
  III     & 1   & 0   & -1	   \\  \hline
\end{tabular}

Det = $3$
\end{minipage}
\begin{minipage}[b]{2.5 in}
\begin{tabular}{|c | c | c | c |c | }
 \hline 
 \multicolumn{4}{|l|}{Matrix at height 4} \\ \hline \hline
    & I & II  & III  \\ \hline   
  I       & -1   & 0   & 0    \\  \hline
  II      & 0   & -2   & 0 	 \\  \hline
  III     & 1   & -1   & -1	   \\  \hline
\end{tabular}

Det has one nonzero term.
\end{minipage}

\begin{minipage}[b]{2.5 in}
\begin{tabular}{|c | c | c | c |c | }
 \hline 
 \multicolumn{4}{|l|}{Matrix at height 5} \\ \hline \hline
    & I & II  & III  \\ \hline   
  I       & -1   & 0   & -1    \\  \hline
  III     & 0   & -1   & -2	   \\  \hline
\end{tabular}

Det of (I,II)  has one nonzero term.
\end{minipage}
}

\newpage

\subsection{The root system $E_6$}  This describes the roots and structure constants for $E_6$.

 {\tiny
 \begin{tabular}{|l|l|l|l|l|l|l|l|} \hline
  \multicolumn{7}{|c|}{Type $E_6$}      \\ \hline 
                &  VI      &  III     &   I     &  V      &  IV     &   II    \\ \hline
 $\Phi_1$       & $000001$   & $001000$   & $100000$  & $000010$  & $000100$  &  $010000$\\
 $\Phi_2$       & $000011$   & $001100$   & $101000$  & $000110$  & $010100$  &            \\
 $\Phi_3$       & $000111$   & $011100$   & $101100$  & $001110$  & $010110$  &            \\
 $\Phi_4$       & $001111$   & $011110$   & $111100$  & $101110$  & $010111$  &            \\
 $\Phi_5$       & $011111$   & $011210$   & $111110$  & $101111$  &            &            \\
 $\Phi_6$       & $111111$   & $011211$   & $111210$  &            &            &            \\
 $\Phi_7$       & $111211$   & $011221$   & $112210$  &            &            &            \\
 $\Phi_8$       & $111221$   &             & $112211$  &            &            &            \\
 $\Phi_9$       & $112221$   &             &            &            &            &            \\
 $\Phi_{10}$    & $112321$   &             &            &            &            &            \\
 $\Phi_{11}$    & $122321$   &             &            &            &            &            \\ \hline
 \end{tabular}

\begin{minipage}[b]{3.5in}
\begin{tabular}{|c | c | c | c |c | c | c   |}
 \hline 
 \multicolumn{7}{|l|}{Matrix at height 1} \\ \hline \hline
  &  I & II & III & IV  & V  & VI      \\ \hline   
  I     & 1  & 0  & -1   & 0   & 0  & 0       \\  \hline
  III   & 0  & 0  & 1   & -1   & 0  & 0       \\  \hline
  IV    & 0  & 1  & 0   & -1   & 0  & 0   	  \\  \hline
  V     & 0  & 0  & 0   & 1   & -1  & 0    	  \\  \hline
  VI    & 0  & 0  & 0   & 0   & 1  & -1    	  \\  \hline
%  VII   & 0  & 0  & 0   & 0   & 0  & -1   	  \\  \hline
\end{tabular}

Det(I,III, IV, V, VI) has one nonzero term.
\end{minipage}
\begin{minipage}[b]{3.5in}
\begin{tabular}{|c | c | c | c |c | c | }
 \hline 
 \multicolumn{6}{|l|}{Matrix at height 2} \\ \hline \hline
  &  I  & III & IV  & V  & VI      \\ \hline   
  I     & 1   & -1   & 0   & 0  & 0      \\  \hline
  III   & 0   & -1   & -1   & 0  & 0       \\  \hline
  IV    & 0   & 0   & 1   & -1  & 0  	  \\  \hline
  V     & 0   & 1   & 0   & -1  & 0   	  \\  \hline
  VI    & 0   & 0   & 0   & 1  & -1    	  \\  \hline
\end{tabular}

Det =$-2$
\end{minipage}

\begin{minipage}[b]{3.5in}
\begin{tabular}{|c | c | c | c |c | c |  }
 \hline 
 \multicolumn{6}{|l|}{Matrix at height 3} \\ \hline \hline
  &  I  & III & IV  & V  & VI      \\ \hline   
  I     & -1   & -1   & 0   & 0  & 0       \\  \hline
  III   & 0   & 1   & -1   & -1  & 0       \\  \hline
  IV    & 0   & 0   & 1   & 0  & -1     \\  \hline
  V     & 1   & 0   & 0   & -1  & 0    	  \\  \hline
  VI    & 0   & 0   & 0   & 1  & -1    	  \\  \hline
\end{tabular}

Det=3
\end{minipage}
\begin{minipage}[b]{3.5in}
\begin{tabular}{|c | c | c | c |c | c |   }
 \hline 
 \multicolumn{6}{|l|}{Matrix at height 4} \\ \hline \hline
  &  I  & III & IV  & V  & VI     \\ \hline   
  I     & 1   & -1   & 0   & -1  & 0       \\  \hline
  III   & 0   & -1   & 0   & 0  & 0      \\  \hline
  V     & 0   & 0   & 0   &1  & -1   	  \\  \hline
  VI    & 0   & 1  & -1   & 0  & -1    	  \\  \hline
\end{tabular}

Det(I,III,V,VI) has one nonzero term.
\end{minipage}

\begin{minipage}[b]{3.5in}
\begin{tabular}{|c | c | c |c | c |  }
 \hline 
 \multicolumn{5}{|l|}{Matrix at height 5} \\ \hline \hline
  &  I  & III   & V  & VI     \\ \hline   
  I     & -1   & -1     & 0  & 0       \\  \hline
  III   & 0   & 1     & 0  & -1       \\  \hline
  VI    & 1   & 0     & -1  & -1    	  \\  \hline
\end{tabular}

Det(I,III,V) has one nonzero term.
\end{minipage}
\begin{minipage}[b]{3.5in}
\begin{tabular}{|c | c | c |c |  }
 \hline 
 \multicolumn{4}{|l|}{Matrix at height 6} \\ \hline \hline
  &  I  & III     & VI      \\ \hline   
  I     & -1   & 0       & 0     \\  \hline
  III   & 0   & -1       & 0      \\  \hline
  VI    & 1   & -1        & -1    	  \\  \hline
\end{tabular}

Det has one nonzero term.
\end{minipage}

\begin{minipage}[b]{3.5in}
\begin{tabular}{|c | c | c | c |  }
 \hline 
 \multicolumn{4}{|l|}{Matrix at height 7} \\ \hline \hline
  &  I  & III     & VI     \\ \hline   
  I     &  1  & 0       & 0       \\  \hline
  VI    &  0  & -1      & -1    	  \\  \hline
\end{tabular}

Det(I,VI) has one nonzero term.
\end{minipage}
}

\newpage
\subsection{The root system $E_7$} This describes the roots and structure constants for $E_7$.

{\tiny
 \begin{tabular}{|l|l|l|l|l|l|l|l|l|} \hline
  \multicolumn{8}{|c|}{Type $E_7$}      \\ \hline 
  Height &  VII    &  VI      &  III     &   I     &  V      &  IV     &   II     \\ \hline
 1       & $0000001$  & $0000010$   & $0010000$   & $1000000$  & $0000100$  & $0001000$  &  $0100000$\\
 2       & $0000011$  & $0000110$   & $0011000$   & $1010000$  & $0001100$  & $0101000$  &            \\
 3       & $0000111$  & $0001110$   & $0111000$   & $1011000$  & $0011100$  & $0101100$  &            \\
 4       & $0001111$  & $0011110$   & $0111100$   & $1111000$  & $1011100$  & $0101110$  &            \\
 5       & $0011111$  & $0111110$   & $0112100$   & $1111100$  & $1011110$  & $0101111$  &            \\
 6       & $0111111$  & $1111110$   & $0112110$   & $1112100$  & $1011111$  &            &            \\
 7       & $0112111$  & $1112110$   & $0112210$   & $1122100$  & $1111111$  &            &            \\
 8       & $1112111$  & $1112210$   & $0112211$   & $1122110$  &            &            &            \\
 9       & $1112211$  & $1122210$   & $0112221$   & $1122111$  &            &            &            \\
 10      & $1122211$  & $1123210$   & $1112221$   &            &            &            &            \\
 11      & $1123211$  & $1223210$   & $1122221$   &            &            &            &            \\
 12      & $1123221$  & $1223211$   &             &            &            &            &            \\
 13      & $1123321$  & $1223221$   &             &            &            &            &            \\
 14      & $1223321$  &             &             &            &            &            &            \\
 15      & $1224321$  &             &             &            &            &            &            \\
 16      & $1234321$  &             &             &            &            &            &            \\
 17      & $2234321$  &             &             &            &            &            &            \\ \hline

 \end{tabular}

\begin{minipage}[b]{3.5 in}
\begin{tabular}{|c | c | c | c |c | c | c | c | c |}
 \hline 
 \multicolumn{8}{|l|}{Matrix at height 1} \\ \hline \hline
  &  I & II & III & IV  & V  & VI  & VII    \\ \hline   
  I     & 1  & 0  & -1   & 0   & 0  & 0   & 0    \\  \hline
  III   & 0  & 0  & 1   & -1   & 0  & 0   & 0    \\  \hline
  IV    & 0  & 1  & 0   & -1   & 0  & 0   &	0	  \\  \hline
  V     & 0  & 0  & 0   & 1   & -1  & 0   & 0 	  \\  \hline
  VI    & 0  & 0  & 0   & 0   & 1  & -1   & 0 	  \\  \hline
  VII   & 0  & 0  & 0   & 0   & 0  & 1   & -1	  \\  \hline
\end{tabular}

Det(I,II,III,V,VI,VII) has one nonzero term.
\end{minipage}
\begin{minipage}[b]{3.5in}
\begin{tabular}{|c | c | c | c |c | c | c |  }
 \hline 
 \multicolumn{7}{|l|}{Matrix at height 2} \\ \hline \hline
  &  I  & III & IV  & V  & VI  & VII    \\ \hline   
  I     & 1   & -1   & 0   & 0  & 0   & 0  \\  \hline
  III   & 0   & -1   & -1   & 0  & 0   & 0    \\  \hline
  IV    & 0   & 0   & 1   & -1  & 0   &	0	  \\  \hline
  V     & 0   & 1   & 0   & -1  & 0   & 0  \\  \hline
  VI    & 0   & 0   & 0   & 1  & -1   & 0 	 \\  \hline
  VII   & 0   & 0   & 0   & 0  & 1   & -1	  \\  \hline
\end{tabular}

Det=2
\end{minipage}

\begin{minipage}[b]{3.5in}
\begin{tabular}{|c | c | c | c |c | c | c |  }
 \hline 
 \multicolumn{7}{|l|}{Matrix at height 3} \\ \hline \hline
  &  I  & III & IV  & V  & VI  & VII    \\ \hline   
  I     & -1   & -1   & 0   & 0  & 0   & 0   \\  \hline
  III   & 0   & 1   & -1   & -1  & 0   & 0  \\  \hline
  IV    & 0   & 0   & 1   & 0  & -1   &	0	  \\  \hline
  V     & 1   & 0   & 0   & -1  & 0   & 0  \\  \hline
  VI    & 0   & 0   & 0   & 1  & -1   & 0 	 \\  \hline
  VII   & 0   & 0   & 0   & 0  & 1   & -1	  \\  \hline
  
\end{tabular}

Det=3
\end{minipage}
\begin{minipage}[b]{3.5in}
\begin{tabular}{|c | c | c | c |c | c | c | c | }
 \hline 
 \multicolumn{7}{|l|}{Matrix at height 4} \\ \hline \hline
  &  I  & III & IV  & V  & VI  & VII    \\ \hline   
  I     & 1   & -1   & 0   & -1  & 0   & 0   \\  \hline
  III   & 0   & -1   & 0   & 0  & 0   & 0    \\  \hline
  IV    & 0   & 0   & 1   & 0  & 0   & -1	 \\  \hline
  V     & 0   & 0   & 0   & 1  & -1   & 0 	   \\  \hline
  VI    & 0   & 1  & -1   & 0  & -1   & 0 	 \\  \hline
  VII   & 0   & 0   & 0   & 0  & 1   & -1	  \\  \hline
\end{tabular}

Det=-2
\end{minipage}

\begin{minipage}[b]{3.5in}
\begin{tabular}{|c | c | c | c |c | c | c | }
 \hline 
 \multicolumn{7}{|l|}{Matrix at height 5} \\ \hline \hline
  &  I  & III & IV  & V  & VI  & VII    \\ \hline   
  I     & -1   & -1   & 0   & 0  & 0   & 0    \\  \hline
  III   & 0   & 1   & 0   & 0  & -1   & 0    \\  \hline
  V     & 0   & 0   & 0   & 1  & 0   & -1 	 \\  \hline
  VI    & 1   & 0   & 0   & -1  & -1   & 0 	 \\  \hline
  VII   & 0   & 0   & -1   & 0  & 1   & -1	 \\  \hline
\end{tabular}

Det(I,III,V,VI,VII)=-3
\end{minipage}
\begin{minipage}[b]{3.5in}
\begin{tabular}{|c | c | c | c |c | c |  }
 \hline 
 \multicolumn{6}{|l|}{Matrix at height 6} \\ \hline \hline
  &  I  & III   & V  & VI  & VII    \\ \hline   
  I     & -1   & 0      & 0  & 0   & 0    \\  \hline
  V     & 0   & 0      & -1  & 1   & -1 	  \\  \hline
  VI    & 1   & -1     & 0  & -1   & 0 	  \\  \hline
  VII   & 0   & 1     & 0  & 0   & -1	  \\  \hline
\end{tabular}

Det has one nonzero term.
\end{minipage}

\begin{minipage}[b]{3.5in}
\begin{tabular}{|c | c | c | c |c | c | }
 \hline 
 \multicolumn{6}{|l|}{Matrix at height 7} \\ \hline \hline
  &  I  & III   & V  & VI  & VII    \\ \hline   
  I     &  1  & 0      & 0  & -1   & 0    \\  \hline
  III   &  0  & 1      & 0  & 0   & -1    \\  \hline
  VI    &  0  & -1      & 0  & 1   & 0 	  \\  \hline
  VII   &  0  & 0      & -1  & 1   & -1	  \\  \hline
\end{tabular}

Det(I,III,V,VI) has one nonzero term.
\end{minipage}
\begin{minipage}[b]{3.5in}
\begin{tabular}{|c | c | c | c |c |   }
 \hline 
 \multicolumn{5}{|l|}{Matrix at height 8} \\ \hline \hline
  &  I  & III     & VI  & VII    \\ \hline   
  I     &  1  & 0       & 0   & -1    \\  \hline
  III   &  0  & -1       & 0   & 0    \\  \hline
  VI    &  -1  & 0       & -1   & 0 	  \\  \hline
  VII   &  0  & -1        & 1   & -1	  \\  \hline
\end{tabular}

Det has one nonzero term.
\end{minipage}

\begin{minipage}[b]{2in}
\begin{tabular}{|c | c | c | c |c |  }
 \hline 
 \multicolumn{5}{|l|}{Matrix at height 9} \\ \hline \hline
  &  I  & III     & VI  & VII    \\ \hline   
  III   &  0  & -1        & 0   & -1    \\  \hline
  VI    &  0  & 0        & -1   & 0 	  \\  \hline
  VII   &  -1  & 0        & 1   & -1	  \\  \hline
\end{tabular}

Det(III,VI,VII) has one term.
\end{minipage}
\begin{minipage}[b]{2in}
\begin{tabular}{|c | c | c | c |  }
 \hline 
 \multicolumn{4}{|l|}{Matrix at height 10} \\ \hline \hline
    & III     & VI  & VII    \\ \hline   
  III     & -1        & 0   & 0    \\  \hline
  VI      & 0        & -1   & 0 	  \\  \hline
  VII     & 0        & 1   & -1	  \\  \hline
\end{tabular}

Det has one nonzero term.
\end{minipage}
\begin{minipage}[b]{2in}
\begin{tabular}{|c | c | c | c |  }
 \hline 
 \multicolumn{4}{|l|}{Matrix at height 11} \\ \hline \hline
    & III    & VI  & VII    \\ \hline   
  III     & 1       & 0   & 0    \\  \hline
  VI      & 0        & 1   & -1 	 \\  \hline
  VII     & -1        & 0   & -1	  \\  \hline
\end{tabular}

Det has one nonzero term.
\end{minipage}
}

\newpage
\subsection{The root system $E_8$} This describes the roots and structure constants for $E_8$.

{\tiny
 \begin{tabular}{|l|l|l|l|l|l|l|l|l|} \hline
  \multicolumn{9}{|c|}{Type $E_8$}      \\ \hline 
  Height &  VII       &  VI    &  III   &  IV   &   I   &  VIII   &  V &  II   \\ \hline
 1  & $00000010$ &  $00000100$ & $00100000$ & $00010000$ & $10000000$  & $00000001$  & $00001000$  &  $01000000$\\
 2  & $00000110$ & $00001100$ & $00110000$  & $01010000$ & $10100000$  & $00000011$  & $00011000$  & \\
 3  & $00001110$ & $00011100$  & $01110000$  & $01011000$ & $10110000$  & $00000111$ & $00111000$  & \\
 4  & $00011110$ & $00111100$  & $01111000$  & $01011100$ & $11110000$  & $00001111$ & $10111000$  & \\
 5  & $00111110$ & $01111100$  & $01121000$  & $01011110$ & $11111000$  & $00011111$ & $10111100$  & \\
 6  & $01111110$ & $11111100$ & $01121100$  & $01011111$  & $11121000$  & $00111111$ & $10111110$  & \\
 7  & $01121110$ & $11121100$ & $01122100$  & $01111111$  & $11221000$  & $10111111$ & $11111110$  & \\
 8  & $11121110$  & $11122100$ & $01122110$ & $01121111$  & $11221100$  & $11111111$ &  & \\
 9  & $11122110$  & $11222100$ & $01122210$ & $01122111$  & $11221110$  & $11121111$ &  & \\
 10 & $11222110$  & $11232100$ & $11122210$ & $01122211$  & $11221111$  & $11122111$ &  & \\
 11 & $11232110$  &  $12232100$ & $11222210$ & $01122221$ & $11222111$  & $11122211$  &  & \\
 12 & $11232210$  & $12232110$ &  $11222211$ & $11122221$ & $11232111$ &  &  & \\
 13 & $11233210$  & $12232210$  & $11232211$ & $11222221$ & $12232111$  &  &  & \\
 14 & $12233210$  & $12232211$  & $11233211$ & $11232221$ &  &  &  & \\
 15 & $12243210$  & $12233211$ &  $11233221$ & $12232221$ &  &  &  & \\
 16 & $12343210$  & $12243211$ &  $11233321$ & $12233221$ &  &  &  & \\
 17 & $22343210$  & $12343211$  & $12233321$ & $12243221$ &  &  &  & \\
 18 & $22343211$  & $12343221$ &  $12243321$ &  &  &  &  & \\
 19 & $22343221$  & $12343321$ &  $12244321$ &  &  &  &  & \\
 20 & $22343321$  & $12344321$ &             &  &  &  &  & \\
 21 & $22344321$  & $12354321$ &             &  &  &  &  & \\
 22 & $22354321$  & $13354321$ &             &  &  &  &  & \\
 23 & $22454321$  & $23354321$ &             &  &  &  &  & \\
 24 & $23454321$  &  &  &  &  &  &  & \\
 25 & $23464321$  &  &  &  &  &  &  & \\
 26 & $23465321$  &  &  &  &  &  &  & \\
 27 & $23465421$  &  &  &  &  &  &  & \\
 28 & $23465431$  &  &  &  &  &  &  & \\
 29 & $23465432$  &  &  &  &  &  &  & \\ \hline
 \end{tabular}

\begin{minipage}[b]{3.5in}
\begin{tabular}{|c | c | c | c |c | c | c | c | c |}
 \hline 
 \multicolumn{9}{|l|}{Matrix at height 1} \\ \hline \hline
  &  I & II & III & IV  & V  & VI  & VII & VIII   \\ \hline   
  I     & 1  & 0  & -1   & 0   & 0  & 0   & 0   & 0 \\  \hline
  III   & 0  & 0  & 1   & -1   & 0  & 0   & 0   & 0 \\  \hline
  IV    & 0  & 1  & 0   & -1   & 0  & 0   &	0	 & 0 \\  \hline
  V     & 0  & 0  & 0   & 1   & -1  & 0   & 0 	 & 0 \\  \hline
  VI    & 0  & 0  & 0   & 0   & 1  & -1   & 0 	 & 0 \\  \hline
  VII   & 0  & 0  & 0   & 0   & 0  & 1   & -1	 & 0 \\  \hline
  VIII  & 0  & 0  & 0   & 0   & 0  & 0   & 1	 & -1  \\  \hline
\end{tabular}

Det(I,II,III,IV,V,VI,VII) has one nonzero term.
\end{minipage}
\begin{minipage}[b]{3.5in}
\begin{tabular}{|c | c | c | c |c | c | c | c | }
 \hline 
 \multicolumn{8}{|l|}{Matrix at height 2} \\ \hline \hline
  &  I  & III & IV  & V  & VI  & VII & VIII   \\ \hline   
  I     & 1   & -1   & 0   & 0  & 0   & 0   & 0 \\  \hline
  III   & 0   & -1   & -1   & 0  & 0   & 0   & 0 \\  \hline
  IV    & 0   & 0   & 1   & -1  & 0   &	0	 & 0 \\  \hline
  V     & 0   & 1   & 0   & -1  & 0   & 0 	 & 0 \\  \hline
  VI    & 0   & 0   & 0   & 1  & -1   & 0 	 & 0 \\  \hline
  VII   & 0   & 0   & 0   & 0  & 1   & -1	 & 0 \\  \hline
  VIII  & 0   & 0   & 0   & 0  & 0   & 1	 & -1  \\  \hline
\end{tabular}

Det=2
\end{minipage}

\begin{minipage}[b]{3.5in}
\begin{tabular}{|c | c | c | c |c | c | c | c | }
 \hline 
 \multicolumn{8}{|l|}{Matrix at height 3} \\ \hline \hline
  &  I  & III & IV  & V  & VI  & VII & VIII   \\ \hline   
  I     & -1   & -1   & 0   & 0  & 0   & 0   & 0 \\  \hline
  III   & 0   & 1   & -1   & -1  & 0   & 0   & 0 \\  \hline
  IV    & 0   & 0   & 1   & 0  & -1   &	0	 & 0 \\  \hline
  V     & 1   & 0   & 0   & -1  & 0   & 0 	 & 0 \\  \hline
  VI    & 0   & 0   & 0   & 1  & -1   & 0 	 & 0 \\  \hline
  VII   & 0   & 0   & 0   & 0  & 1   & -1	 & 0 \\  \hline
  VIII  & 0   & 0   & 0   & 0  & 0   & 1	 & -1  \\  \hline
\end{tabular}

Det=3
\end{minipage}
\begin{minipage}[b]{3.5in}
\begin{tabular}{|c | c | c | c |c | c | c | c | }
 \hline 
 \multicolumn{8}{|l|}{Matrix at height 4} \\ \hline \hline
  &  I  & III & IV  & V  & VI  & VII & VIII   \\ \hline   
  I     & 1   & -1   & 0   & -1  & 0   & 0   & 0 \\  \hline
  III   & 0   & -1   & 0   & 0  & 0   & 0   & 0 \\  \hline
  IV    & 0   & 0   & 1   & 0  & 0   & -1	 & 0 \\  \hline
  V     & 0   & 0   & 0   & 1  & -1   & 0 	 & 0 \\  \hline
  VI    & 0   & 1  & -1   & 0  & -1   & 0 	 & 0 \\  \hline
  VII   & 0   & 0   & 0   & 0  & 1   & -1	 & 0 \\  \hline
  VIII  & 0   & 0   & 0   & 0  & 0   & 1	 & -1  \\  \hline
\end{tabular}

Det=-2
\end{minipage}

\begin{minipage}[b]{3.5in}
\begin{tabular}{|c | c | c | c |c | c | c | c | }
 \hline 
 \multicolumn{8}{|l|}{Matrix at height 5} \\ \hline \hline
  &  I  & III & IV  & V  & VI  & VII & VIII   \\ \hline   
  I     & -1   & -1   & 0   & 0  & 0   & 0   & 0 \\  \hline
  III   & 0   & 1   & 0   & 0  & -1   & 0   & 0 \\  \hline
  IV    & 0   & 0   & 1   & 0  & 0   & 0	 & -1 \\  \hline
  V     & 0   & 0   & 0   & 1  & 0   & -1 	 & 0 \\  \hline
  VI    & 1   & 0   & 0   & -1  & -1   & 0 	 & 0 \\  \hline
  VII   & 0   & 0   & -1   & 0  & 1   & -1	 & 0 \\  \hline
  VIII  & 0   & 0   & 0   & 0  & 0   & 1	 & -1  \\  \hline
\end{tabular}

Det=-5
\end{minipage}
\begin{minipage}[b]{3.5in}
\begin{tabular}{|c | c | c | c |c | c | c | c | }
 \hline 
 \multicolumn{8}{|l|}{Matrix at height 6} \\ \hline \hline
  &  I  & III & IV  & V  & VI  & VII & VIII   \\ \hline   
  I     & -1   & 0   & 0   & 0  & 0   & 0   & 0 \\  \hline
  III   & 0   & -1   & 0   & 0  & 0   & 0   & 0 \\  \hline
  IV    & 0   & 0   & -1   & 0  & 0   & 1	 & -1 \\  \hline
  V     & 0   & 0   & 0   & -1  & 1   & -1 	 & 0 \\  \hline
  VI    & 1   & -1   & 0   & 0  & -1   & 0 	 & 0 \\  \hline
  VII   & 0   & 1   & 0   & 0  & 0   & -1	 & 0 \\  \hline
  VIII  & 0   & 0   & 0   & 1  & 0   & 0	 & -1  \\  \hline
\end{tabular}

Det has one nonzero term.
\end{minipage}

\begin{minipage}[b]{3.5in}
\begin{tabular}{|c | c | c | c |c | c | c | c | }
 \hline 
 \multicolumn{8}{|l|}{Matrix at height 7} \\ \hline \hline
  &  I  & III & IV  & V  & VI  & VII & VIII   \\ \hline   
  I     &  1  & 0   & 0   & 0  & -1   & 0   & 0 \\  \hline
  III   &  0  & 1   & 0   & 0  & 0   & -1   & 0 \\  \hline
  IV    &  0  & 0   & -1   & 0  & 0   & 1	 & 0 \\  \hline
  VI    &  0  & -1   & 0   & 0  & 1   & 0 	 & 0 \\  \hline
  VII   &  0  & 0   & 0   & -1  & 1   & -1	 & 0 \\  \hline
  VIII  &  0  & 0   & -1   & 1  & 0   & 0	 & -1  \\  \hline
\end{tabular}

Det(I,III,IV,V,VI,VIII) has one nonzero term.
\end{minipage}
\begin{minipage}[b]{3.5in}
\begin{tabular}{|c | c | c | c |c | c | c |  }
 \hline 
 \multicolumn{7}{|l|}{Matrix at height 8} \\ \hline \hline
  &  I  & III & IV    & VI  & VII & VIII   \\ \hline   
  I     &  1  & 0   & 0     & 0   & -1   & 0 \\  \hline
  III   &  0  & -1   & 0     & 0   & 0   & 0 \\  \hline
  IV    &  0  & 1   & -1     & 0   & 0	 & 0 \\  \hline
  VI    &  -1  & 0   & 0     & -1   & 0 	 & 0 \\  \hline
  VII   &  0  & -1   & 0     & 1   & -1	 & 0 \\  \hline
  VIII  &  0  & 0   & -1     & 0   & 1	 & -1  \\  \hline
\end{tabular}

Det = -2
\end{minipage}

\begin{minipage}[b]{3.5in}
\begin{tabular}{|c | c | c | c |c | c | c |  }
 \hline 
 \multicolumn{7}{|l|}{Matrix at height 9} \\ \hline \hline
  &  I  & III & IV    & VI  & VII & VIII   \\ \hline   
  I     &  1  & 0   & 0     & 0   & 0   & -1 \\  \hline
  III   &  0  & -1   & 0     & 0   & -1   & 0 \\  \hline
  IV    &  0  & 1   & -1     & 0   & 0	 & 0 \\  \hline
  VI    &  0  & 0   & 0     & -1   & 0 	 & 0 \\  \hline
  VII   &  -1  & 0   & 0     & 1   & -1	 & 0 \\  \hline
  VIII  &  0  & 0   & -1     & 0   & 1	 & -1  \\  \hline
\end{tabular}

Det = -3
\end{minipage}
\begin{minipage}[b]{3.5in}
\begin{tabular}{|c | c | c | c |c | c | c |  }
 \hline 
 \multicolumn{7}{|l|}{Matrix at height 10} \\ \hline \hline
  &  I  & III & IV    & VI  & VII & VIII   \\ \hline   
  I     &  -1  & 0   & 0     & 0   & 1   & -1 \\  \hline
  III   &  0  & -1   & 0     & 0   & 0   & 0 \\  \hline
  IV    &  0  & 0   & -1     & 0   & 0	 & 0 \\  \hline
  VI    &  0  & 0   & 0     & -1   & 0 	 & 0 \\  \hline
  VII   &  0  & 0   & 0     & 1   & -1	 & 0 \\  \hline
  VIII  &  0  & 1   & -1     & 0   & 0	 & -1  \\  \hline
\end{tabular}

Det has one nonzero term.
\end{minipage}

\begin{minipage}[b]{3.5in}
\begin{tabular}{|c | c | c | c |c | c | c |  }
 \hline 
 \multicolumn{7}{|l|}{Matrix at height 11} \\ \hline \hline
  &  I  & III & IV    & VI  & VII & VIII   \\ \hline   
  I     &  -1  & 0   & 0     & 0   & 1   & 0 \\  \hline
  III   &  -1  & 1   & 0     & 0   & 0   & -1 \\  \hline
  IV    &  0  & 0   & -1     & 0   & 0	  & -1 \\  \hline
  VI    &  0  & 0   & 0     & 1   & -1 	& 0 \\  \hline
  VII   &  0  & -1   & 0     & 0   & -1	 & 0 \\  \hline
\end{tabular}

Det(I,III,IV,VI,VII) = -2
\end{minipage}
\begin{minipage}[b]{3.5in}
\begin{tabular}{|c | c | c | c |c | c |  }
 \hline 
 \multicolumn{6}{|l|}{Matrix at height 12} \\ \hline \hline
  &  I  & III & IV    & VI  & VII  \\ \hline   
  I     &  -1  & 0   & 0     & 1   & 0    \\  \hline
  III   &  -1  & -1   & 0     & 0   &1    \\  \hline
  IV    &  0  & -1   & -1     & 0   & 0	   \\  \hline
  VI    &  0  & 0   & 0     & -1   & -1 	 \\  \hline
  VII   &  0  & 0   & 0     & 0   & -1	   \\  \hline
\end{tabular}

Det has one nonzero term.
\end{minipage}

\begin{minipage}[b]{3.5in}
\begin{tabular}{|c | c | c | c |c | c | }
 \hline 
 \multicolumn{6}{|l|}{Matrix at height 13} \\ \hline \hline
  &  I  & III & IV    & VI  & VII  \\ \hline   
  III   &  0  & -1   & 0     & 0   & 1    \\  \hline
  IV    &  0  & -1   & -1     & 0   & 0	   \\  \hline
  VI    &  -1  & -1   & 0     & 1   & 0 	 \\  \hline
  VII   &  0  & 0   & 0     & -1   & -1	   \\  \hline
\end{tabular}

Det(I,III,IV,VI) has one nonzero term.
\end{minipage}
\begin{minipage}[b]{3.5in}
\begin{tabular}{|c | c | c | c |c | }
 \hline 
 \multicolumn{5}{|l|}{Matrix at height 14} \\ \hline \hline
    & III & IV    & VI  & VII  \\ \hline   
  III     & -1   & -1     & 0   & 0    \\  \hline
  IV      & 0   & -1     & -1   & 0	   \\  \hline
  VI      & -1   & 0     & -1   & 1 	 \\  \hline
  VII     & 0   & 0     & 0   & -1	   \\  \hline
\end{tabular}

Det = 2
\end{minipage}

\begin{minipage}[b]{3.5in}
\begin{tabular}{|c | c | c | c |c | c | }
 \hline 
 \multicolumn{5}{|l|}{Matrix at height 15} \\ \hline \hline
    & III & IV    & VI  & VII  \\ \hline   
  III     & -1   & 0     & 0   & 0    \\  \hline
  IV      & -1   & -1     & -1   & 0	   \\  \hline
  VI      & 0   & 0     & -1   & -1 	 \\  \hline
  VII     & 0   & 0     & 0   & -1	   \\  \hline
\end{tabular}

Det has one nonzero term.
\end{minipage}
\begin{minipage}[b]{3.5in}
\begin{tabular}{|c | c | c | c |c | c | }
 \hline 
 \multicolumn{5}{|l|}{Matrix at height 16} \\ \hline \hline
    & III & IV    & VI  & VII  \\ \hline   
  III     & -1   & -1     & 0   & 0    \\  \hline
  IV      & 0   & -1     & -1   & 0	   \\  \hline
  VI      & 0   & 0     & -1   & -1 	 \\  \hline
  VII     & 0   & 0     & 0   & -1	   \\  \hline
\end{tabular}

Det has one nonzero term.
\end{minipage}

\begin{minipage}[b]{3.5in}
\begin{tabular}{|c | c | c | c |c | c | }
 \hline 
 \multicolumn{5}{|l|}{Matrix at height 17} \\ \hline \hline
    & III & IV    & VI  & VII  \\ \hline   
  III     & -1   & -1     & 0   & 0    \\  \hline
  VI      & 0   & -1     & -1   & 0 	 \\  \hline
  VII     & 0   & 0     & -1   & -1	   \\  \hline
\end{tabular}

Det(III,IV,VI) has one nonzero term.
\end{minipage}
\begin{minipage}[b]{3.5in}
\begin{tabular}{|c | c | c | c |c | }
 \hline 
 \multicolumn{4}{|l|}{Matrix at height 18} \\ \hline \hline
    & III & VI  & VII  \\ \hline   
  III     & -1   & 0   & 0    \\  \hline
  VI      & -1   & -1   & 0 	 \\  \hline
  VII     & 0   & -1   & -1	   \\  \hline
\end{tabular}

Det has one nonzero term.
\end{minipage}

\begin{minipage}[b]{3.5in}
\begin{tabular}{|c | c | c | c |c | }
 \hline 
 \multicolumn{4}{|l|}{Matrix at height 19} \\ \hline \hline
    & III & VI  & VII  \\ \hline   
  VI      & -1   & -1   & 0 	 \\  \hline
  VII     & 0   & -1   & -1	   \\  \hline
\end{tabular}

Det(III,VI) has one nonzero term.
\end{minipage}
}

We omit tables for heights over $19$ since there are only two columns in those cases, so the matrix $\matr_i$ has full rank by Lemma \ref{2x2}.


\begin{thebibliography}{9999999999999}

\bibitem[Aky81]{Aky81} E.~Akyildiz, \emph{Bruhat decomposition via Gm-action,} Bull. Acad. Polon. Sci. Ser. Sci. Math. 28,
no. 11-12, 541--547 (1981)

\bibitem[Bay-Har10]{BH10}D. ~Bayegan and M.~ Harada. \textit{A Giambelli formula for the $S_1$-equivariant cohomology of type A Peterson varieties,} (preprint)arXiv:1012.4053. To be published in Involve.

\bibitem[Bay-Har12]{BH12} D.~Bayegan and M.~Harada. \textit{Poset pinball, the dimension pair algorithm, and type A regular nilpotent Hessenberg varieties.} ISRN Geometry, Volume 2012, Article ID: 254235 (2012), doi:10.5402/2012/254235.
 

\bibitem[Bia-Bir76]{BB76} A.~Bialynicki-Birula, \emph{ Some properties of the decompositions of algebraic varieties determined by actions of a torus.} Bull. Acad. Polon. Sci. Ser. Sci. Math. Astronom. Phys. {\bf24} (1976) 667-674.

\bibitem[Bil99]{Bil99} S.~Billey, \emph{Kostant polynomials and the cohomology ring of G/B.} Duke Math. J., \textbf{96} (1999)205–224.

\bibitem[Bil-Lak00]{Billey.Lakshmibai}
S. Billey and V. Lakshmibai, \emph{Singular loci of Schubert varieties}, Progr. Math. {\bf 182}(2000),
Birkhauser, Boston.

\bibitem[Bri-Car04]{Bri-Car04}M.~Brion and J.~B.~Carrell, {\sl The equivariant cohomology ring of regular varieties,}
Michigan Math. J. Volume 52, Issue 1 (2004), 189-203.

\bibitem[Bjo-Bre03]{BjoBre03} A.~Bjorner and F.~Brenti, \emph{Combinatorics of Coxeter Groups.} Springer, Berlin (2003).


\bibitem[Car00]{Car00}
J.~B.~Carrell, \emph{Torus actions and cohomology}, Encyclopedia of  Mathematical Sciences. {\bf 131}(2000),
Springer-Verlag, New York.


\bibitem[Car-Gor83]{CG83}
J.~B.~Carrell and R.~Goresky, \emph{A decomposition theorem for the integral homology of a variety}, Invent. Math. {\bf 3}(1983),
367--381.

\bibitem[Car89]{Car89} R. Carter, {Simple groups of Lie type.}  Wiley. London (1989).

\bibitem[Che94]{Che} C.~Chevalley, {\sl Sur les decompositions cellulaires des
espaces $G/B$,\/} Proc.\ Symp.\ Pure Math.\ {\bf 56} 1994, Part I, 1-25.

\bibitem[Col-McG93]{CM93}  D. Collingwood, W. M. McGovern, { Nilpotent orbits in semisimple Lie algebras.} Van Nostrand Reinhold Co. New York (1993).

\bibitem[DeM-Pro-Sha92]{DPS} F. De Mari, C. Procesi, M. Shayman,\emph{ Hessenberg varieties}, Trans. Amer. Math. Soc. {\bf 332} (1992) 529--534.

\bibitem[DeM-Sha88]{DeMSha88}F.~De~Mari and M.~A.~Shayman,
\textit{Generalized Eulerian numbers and the topology of the Hessenberg variety of a matrix}, Acta Appl. Math. \textbf{12}
(1988), 213-235.

\bibitem[Din97]{Din97} K.~Ding, {\sl Rook placements and cellular decomposition of partition varieties}, Discrete Math., 170,(1997) 
 no. 1-3,  107 -- 151. 


\bibitem[Fulm97]{Fulm97}J.~Fulman, {\sl Descent identities, Hessenberg varieties, and the Weil Conjectures}, Journal of Combinatorial Theory, Series A
Volume 87, {\bf 2}, 390--397, 1999.

\bibitem[Fult97]{F} W.\ Fulton, 
\newblock{\sl Young tableaux.  With applications to representation theory and geometry.}
\newblock{\em London Math.\ Soc.\ Student Texts} {\bf 35}, Cambridge UP, Cambridge, 1997.

\bibitem[Fun03]{Fun03}F.~Y.~C.~Fung. {\sl On the topology of components of some Springer fibers and their relation to Kazhdan-Lusztig theory.}
Adv.~Math., 178({\bf 2}):244--276, 2003.

\bibitem[Gor-Kot-Mac98]{GKM} M. Goresky, R. Kottwitz and R. MacPherson, \textit{Equivariant cohomology, Koszul duality, and the localization theorem}, Invent. Math., 131:25--83, 1998.

\bibitem[Gor-Mac10]{GorMac10}
M.~Goresky and R.~MacPherson.
\newblock \textit{On the spectrum of the equivariant cohomology ring.}
\newblock { Canad. J. Math.}, {\bf 62}(2):262--283, 2010.

\bibitem[Har-Tym09]{HT09} M.~Harada and J.~Tymoczko, \emph{A positive Monk formula in the $S$-equivariant cohomology of type $A$ Peterson varieties}, preprint 2009.
\textsf{arXiv:0908.3517}

\bibitem[Har-Tym10]{HT10} \bysame, \emph{Poset pinball, GKM-compatible subspaces, and Hessenberg varieties}, preprint 2010. \textsf{arXiv:1007.2750} 

\bibitem[Hum64]{H} J.~Humphreys, {\sl Linear Algebraic Groups\/}, Grad.\
Texts in Math.\ 21,
Springer-Verlag, New York, 1964.

\bibitem[Hum72]{H2} \bysame, {\sl Introduction to Lie Algebras and 
Representation Theory\/}, Grad.\
Texts in Math.\ 9,
Springer-Verlag, New York, 1972.


\bibitem[Hum90]{Hum90} \bysame, {\sl Reflection groups and Coxeter groups\/}, Cambridge, 1990.

\bibitem[Ins12]{Ins12} E.~Insko \emph{Equivariant cohomology and local invariants of Hessenberg varieties}, Ph.D. thesis, University of Iowa (2012).

\bibitem[Ins-Yon12]{InsYon12} E.~Insko and A.~Yong, Patch ideals and Peterson varieties, Transform.~Groups {\bf 17} (2012), 1011--1036.

\bibitem[Kna02]{Kna02} A. Knapp, Lie Groups: Beyond an Introduction, Birkh\"auser, Boston, 2002.

\bibitem[Kos96]{K} B.~Kostant,
\newblock{\sl Flag Manifold Quantum Cohomology, the Toda Lattice, and the Representation with Highest Weight $ \rho $\/},
\newblock{ \em Selecta Math.\ (N.\ S.)} {\bf 2} (1996), 43--91.

\bibitem[Man98]{Man} L.~Manivel, {\sl Symmetric Functions, Schubert Polynomials and Degeneracy Loci\/}, SMF/AMS Texts and Monographs {\bf 6}, 1998.

\bibitem[Mbi10]{Mbi10}A.~Mbirika, {\sl A Hessenberg generalization of the Garsia-Procesi basis for the cohomology ring of Springer varieties,} Electronic Journal of Combinatorics, 17({\bf 1}):Research Paper 153, 2010.

\bibitem[Mbi-Tym12]{MT12} A.~Mbirika and J.~Tymoczko, {\sl Generalizing Tanisaki's ideal via ideals of truncated symmetric functions}, Algebraic Combinatorics, DOI:10.1007/s10801-012-0372-2, 2012.

\bibitem[Pet97]{Pet97} D.~Peterson, 
\textit{Quantum cohomology of $G/P$}, Lecture Course, M.~I.~T., Spring Term 1997. 

\bibitem[Pre12]{Pre12}M.~Precup, {\sl Affine pavings of Hessenberg varieties for semi simple groups}, preprint 2012. \textsf{arXiv:1205.3976v2}

\bibitem[Plo-Vav96]{PloVav96} E.~Plotkin and N.~A.~Vavilov, {\sl ChevalleyarXiv:1012.1630 Groups over Commutative Rings: Elementary Calculations}, Acta Applicandae Mathematicae {\bf 45}:73-113,1996.

\bibitem[Sam69]{Sa} H.~Samelson, {\sl Notes on Lie algebras\/}, Van 
Nostrand Reinhold Math.\ Studies {\bf 23}, Van Nostrand Reinhold, New York,
1969.


\bibitem[Rie03]{rietsch} K. Rietsch, \emph{Totally positive Toeplitz matrices and quantum cohomology of partial flag varieties}, J. Amer.~Math.~Soc.,
16(2):363--392 (electronic), 2003.
\bibitem[Rie01]{Rie01} K. Reitsch, \emph{Quantum cohomology rings of Grassmannians and total positivity},
Duke Math. J. 110 (2001), no. 3, 523--553.

\bibitem[Sha-Wac11]{ShaWac11} J.~Shareshian and M.~Wachs. 
\textit{Chromatic quasisymmetric functions and Hessenberg varieties,} preprint. \textsf{arXiv:1106.4287v2}.


\bibitem[Som-Tym06]{ST} E.~Sommers and J.~Tymoczko, {\sl Exponents for
B-stable ideals\/}, Trans.\ Amer.\ Math.\ Soc.\ {\bf 358} (2006), 3493--3509.

\bibitem[Spr76]{Spr76}T.~A.~Springer, {\sl Trigonometric sums, green functions of finite groups and representations of Weyl groups}, Invent. Math. {\bf36} (1976), 173?207.



\bibitem[Spr98]{Spr98}T.~A.~Springer, {\sl Linear Algebraic Groups, second edition}, Birkh\"auser, 1998.

\bibitem[Tym06]{Tym06} J.~Tymoczko, {\sl Linear conditions imposed on flag varieties},Amer.~J.~Math., 128(6):1587--1604, 2006.

\bibitem[Tym06b]{Tym06b} \bysame, {\sl Paving Hessenbergs by affines}, Selecta Mathematica. 

\bibitem[Vav04]{Vav04}N.~A.~Vavilov, {\sl Do it yourself: the structure constants for Lie algebras of type $E_l$,} Journal of Mathematical Sciences, 120(4):1513--1548, 2004.




\end{thebibliography}
\end{document}